\documentclass[amssymb,amsfonts,refcheck,12pt,verbatim,righttag]{amsart}
\usepackage{amssymb}
\usepackage[dvips]{color}
\usepackage{graphicx}

  \hoffset=-1.5cm \errorcontextlines=0 \numberwithin{equation}{section} % \renewcommand{\rm}{\normalshape} %
\theoremstyle{plain}

\newtheorem{thm}{Theorem}[section]
\newtheorem{lem}[thm]{Lemma}
\newtheorem{pro}[thm]{Proposition}
\newtheorem{cor}[thm]{Corollary}
\newtheorem{ex}[thm]{Example}
\newtheorem{de}[thm]{Definition}
\newtheorem{rem}[thm]{Remark}

\def\R {{\Bbb R}}
\def\N {{\Bbb N}}
\def\Z {{\Bbb Z}}
\def\L {{\mathcal L}}
\def\M {{\mathcal M}}

\def\F {{\mathcal F}}
\def\I {{\mathcal I}}
\def\A {{\mathcal A}}

\def\C {{\mathcal C}}

\def\ba{{\bf a}}

\def\bc{{\bf c}}

\def\bq{{\bf q}}

\begin{document}
\baselineskip 15pt
\title{Multifractal formalism for almost all self-affine measures}

\author{Julien Barral}
\address{LAGA (UMR 7539), D\'epartement de Math\'ematiques, Institut Galil\'ee, Universit\'e Paris 13, 99 avenue Jean-Baptiste Cl\'ement , 93430  Villetaneuse, France}
\email{barral@math.univ-paris13.fr}
\author{De-Jun Feng}
\address{
Department of Mathematics\\ The Chinese University of Hong Kong\\ Shatin,  Hong Kong\\ } \email{djfeng@math.cuhk.edu.hk}


\date{}

\begin{abstract}
We conduct the multifractal analysis of self-affine measures for ``almost all'' family of affine maps. Besides partially extending
Falconer's formula of $L^q$-spectrum outside  the range $1< q\leq 2$, the multifractal formalism is also partially verified.
\end{abstract}

\maketitle
\section{Introduction}\label{S-1}

Multifractal analysis  in $\R^d$ aims at describing the geometry of  H\"older singularities for positive Borel measures. Specifically, given a compactly supported positive Borel measure $\mu$ on $\R^d$, one is interested in the Hausdorff dimensions of the level sets
$$
E(\mu, \alpha):=\left\{x\in \R^d:\; \lim_{r\to 0}\frac{\log \mu(B_r(x))}{\log r}=\alpha\right\}\quad (\alpha\ge 0),
$$
where $B_r(x)$ stands for the Euclidean closed ball with radius $r$ centered at $x$.  According to heuristic arguments developed by physicists \cite{FrPa,Halsey}, in presence of self-similarity, one should have
\begin{equation}\label{formalism}
\dim_H E(\mu,\alpha)=\inf_{q\in\R}(\alpha q-\tau(\mu,q)),
\end{equation}
(a negative dimension meaning that $E(\mu,\alpha)=\emptyset$) where $\tau(\mu,\cdot)$ is the $L^q$-spectrum defined as
$$
\tau(\mu,q)=\liminf_{r\to 0}\frac{\log \sup\sum_{j}\mu(B_r(x_j))^q}{\log r},
$$
the supremum being taken over all families of disjoint balls $\{B_r(x_j)\}_j$ with radius $r$ and centers $x_j\in {\rm supp} (\mu)$.

When equality \eqref{formalism} holds, one says that the multifractal formalism holds for $\mu$ at $\alpha$.   So far the multifractal structures of the so-called self-similar measures and more generally self-conformal measures and Gibbs measures on self-conformal sets or conformal repellers have been studied intensively, the validity of the multifractal formalism being observed over wide or even maximal  ranges of exponents $\alpha$ for large subclasses of these measures (see, e.g., \cite{CLP87,Ran89,BMP92,CaMa92,Ols95,Pat97, Pes97,O99, LN99,Shm05,Tes06,Fen07,FenL09,Fen11, JSS11} and the references in \cite{FenL09}).

Much less is known for self-affine measures (to be defined below), except when they are supported on self-affine Sierpinski sponges, or on invariant subsets of such sponges satisfying specification property \cite{Kin95,Ols98,BaMe07,BaFe09,JorRam09}. However, for such measures, one knows that in general the previous multifractal formalism fails, but a refined one  (which is more  related to Hausdorff measures and introduced independently in \cite{B94} and \cite{Ols95}) holds. This is closely related to the fact that the Hausdorff and box dimension of self-affine Sierpinski sponges do not coincide in general.

This paper studies the validity of the multifractal formalism for  ``almost all'' self-affine measures. First of all, let us recall the definition of self-affine measures. Let $S_1, \ldots, S_m: \R^d\to \R^d$ be a family of contracting mappings. Such a family is known as an {\it iterated function system} (IFS). It is well known \cite{Hut81}
that there exists a unique non-empty compact set $F\subset \R^d$, called the {\it attractor} of the IFS,  satisfying
$$F=\bigcup_{i=1}^m S_i(F);$$
Moreover, for any probability vector $(p_1,\ldots, p_m)$ (that is, $p_i>0$ and $\sum_{i=1}^m p_i=1$), there exists a unique Borel probability measure $\mu$ supported on
$F$ such that
$$\mu=\sum_{i=1}^m p_i \mu\circ S_i^{-1}.$$

Here we assume that $S_1,\ldots, S_m$ are affine transformations, in which case, $F$ is called a {\it self-affine set}, and $\mu$ is called a {\it self-affine measure} (self-similar measures correspond to the particular case where the $S_i$ are similitudes). In
particular, we let $S_i=T_i+a_i$ where $T_1, \ldots, T_m$ are non-singular contracting linear mappings and $a_1, \ldots, a_m$ are translation parameters. In \cite{Fal88}
Falconer obtained a formula for the Hausdorff dimension and box-counting dimension of the attractor of the IFS $\{T_i+a_i\}_{i=1}^m$ for  almost all parameter $(a_1,\ldots,
a_m)\in \R^{md}$ in the sense of $md$-dimensional Lebesgue measure, under an additional assumption that $\|T_i\|<1/3$ for all $i$; these dimensions coincide. Later, Solomyak \cite{Sol98} proved that
the  assumption $\|T_i\|<1/3$ for all $i$ can be weakened to $\|T_i\|<1/2$ for all $i$.

In \cite{Fal99}, Falconer obtained the formula of the $L^q$-spectrum of the self-affine measure associated to the IFS $\{T_i+a_i\}_{i=1}^m$ and the probability vector
$(p_1,\ldots, p_m)$ for $1<q\leq 2$ and almost all $(a_1,\ldots, a_m)\in \R^{md}$, still in the sense of $md$-dimensional Lebesgue measure and under the assumption $\|T_i\|<1/2$ for
all $i$.

Before stating Falconer's formula, let us first introduce some definitions.  Let $T$ be a non-singular linear mapping from $\R^d$
to $\R^d$. The {\it singular values} $\alpha_1\geq \alpha_2\geq \cdots\geq \alpha_d$ of $T$ are the positive square roots of the eigenvalues of $T^*T$.
\begin{de}\cite{Fal88}
\label{de-1.1} The singular value function $\phi^s(T)$ is defined for $s>0$ by
$$
\phi^s(T)=\left \{
\begin{array}{ll}
\alpha_1\ldots \alpha_{k-1} \alpha^{s-k+1}_k, & \mbox{ if } k-1<s\leq k\leq d,\\ \mbox{}\\ (\alpha_1\ldots \alpha_d)^{s/d}, & \mbox{ if } s\geq d.
\end{array}
\right.
$$
In particular, set $\phi^0(T)=1$.
\end{de}

Fix a probability vector $(p_1,\ldots, p_m)$ and non-singular contractive linear transformations $T_1$, $\ldots$, $T_m$ from $\R^d$ to $\R^d$. For $\ba=(a_1,\ldots, a_m)\in \R^{md}$,
let $\mu^\ba$ denote the self-affine measure associated with the IFS $\{T_i+a_i\}_{i=1}^m$ and  $(p_1,\ldots, p_m)$. For $k\in \N$, we write for brevity
$\Sigma_k:=\{1,\ldots,m\}^k$. For $I=i_1\ldots i_k\in \Sigma_k$,  denote $T_I:=T_{i_1}\circ\ldots\circ T_{i_k}$, $p_{I}:=p_{i_1}\ldots p_{i_k}$.  For $q\geq 0$, define
\begin{equation}
\label{e-1.1} D(q)=\left\{
\begin{array}{ll}
(q-1) \inf\left\{s\geq 0:\; \sum_{k=1}^\infty \sum_{I\in \Sigma_k}\left(\phi^{s}(T_I)\right)^{1-q} p_I^q<\infty\right\}, &\mbox{ if } 0\leq q<1,\\
\mbox{}\quad  0, &\mbox{ if } q=1,\\
(q-1) \sup\left\{s\geq 0:\; \sum_{k=1}^\infty \sum_{I\in \Sigma_k}\left(\phi^{s}(T_I)\right)^{1-q} p_I^q<\infty\right\}, &\mbox{ if } q>1,
\end{array}
\right.
\end{equation}
and
\begin{equation}
\label{e-1.2} \tau(q)=\left\{
\begin{array}{ll}
(q-1)\min\left\{\frac{D(q)}{q-1}, d\right\},& \mbox{ if }q\neq 1, \\ 0, & \mbox{ if }q=1.
\end{array}
\right.
\end{equation}

We remark that  $D$ and $\tau$ are continuous and piecewise concave  over $(0,\infty)$. More precisely, $D$ and $\tau$ are concave on $(1,\infty)$,   they are also concave on the subintervals $J_k$ of $(0,1)$, $k=0,1,\ldots, d$, where
$J_k=\{q\in (0,1): D(q)/(q-1)\in (k,k+1)\}$ for $k\leq d-1$ and $J_d=\{q\in (0,1): D(q)/(q-1)>d\}$ (see Appendix~\ref{S-8}). Hence  the one-sided derivatives of $D$ and $\tau$ exist for any $q>0$.

Now Falconer's result can be stated as follows.

\begin{thm}[\cite{Fal99}]
\label{thm-1.2} If $\|T_i\|<1/2$ for all $1\leq i\leq m$, then for $\L^{md}$-a.e. $\ba\in \R^{md}$,  the $L^q$-spectrum of $\mu^\ba$ is
$$
\tau(\mu^\ba,q)=\tau(q), \qquad 1<q\leq 2.
$$
\end{thm}

In \cite{Fal99}, Falconer raised some open problems, for instance, how to extend the above formula outside the range $1<q\leq 2$ and how to analyze  the multifractal structure
of $\mu^\ba$ for $\L^{md}$-a.e. $\ba\in \R^{md}$. The main purpose of this paper is to study these problems. % We would like to study the {\it dimension spectrum} of $\mu^\ba$, which is defined as  the Hausdorff dimension of $E(\mu^\ba, \alpha)$.

Our main result  is the following. It will be completed with some results  for $q\ge 2$ in section~\ref{S-6} (see Theorems \ref{corextension}-\ref{thm-6.4}).

\begin{thm}
\label{thm-1.3} Assume that $\|T_i\|<1/2$ for all $1\leq i\leq m$. Let $q\in (0,2)$, $q\neq 1$.
\begin{itemize}
\item[(i)] Let $\alpha\in \{D'(q-), D'(q+)\}$, where $D'(q\pm)$ denote the one-sided derivatives of $D$ at $q$.  Assume that $0<q<1$, $D(q)/(q-1)<1$ and $\alpha q-D(q)\leq 1$. Then for  $\L^{md}$-a.e. $\ba\in \R^{md}$, $\tau(\mu^\ba,q)=\tau(q)=D(q)$, and furthermore, $E(\mu^\ba,\alpha)\neq \emptyset$ and
$$\dim_H E(\mu^\ba,\alpha)=\alpha q-\tau(q).$$

\item[(ii)] Let $q\in (1,2)$.  Assume that  $T_i$ ($i=1,\ldots,m$) are of  the form   $$T_i=\mbox{diag}(t_{i,1}, t_{i,2},\ldots, t_{i,d})$$
 with $\frac{1}{2}>t_{i,1}>t_{i,2}>\ldots>t_{i,d}>0.$
   Assume  furthermore that
$D(q)/(q-1)\in (k, k+1)$ for some integer $0\leq k\leq d-1$ (in this case $\alpha:=D'(q)$ exists) and  $\alpha q-D(q)\in (k, k+1)$.  \begin{itemize}
 \item If $k=0$,  then
 for  $\L^{md}$-a.e. $\ba\in \R^{md}$,
$E(\mu^\ba,\alpha)\neq \emptyset$ and
$$\dim_H E(\mu^\ba,\alpha)=\alpha q-\tau(q).$$
\item If $k>0$, then
 for  $\L^{md}$-a.e. $\ba\in \R^{md}$,
$\underline{E}(\mu^\ba,\alpha)\neq \emptyset$ and
$$\dim_H \underline{E}(\mu^\ba,\alpha)=\alpha q-\tau(q),$$ where
$\underline{E}(\mu^\ba,\alpha):=\left\{x\in \R^d:\; \liminf_{r\to 0}\frac{\log \mu(B_r(x))}{\log r}=\alpha\right\}$.
\end{itemize}
\end{itemize}
\end{thm}

We remark that the functions $\tau$ and $D$ can be determined explicitly in some special case.

 \bigskip

\begin{ex}
\label{ex-1.4} Assume that $T_1=T_2=\ldots=T_m={\rm diag} (t_1, t_2,\ldots, t_d)$ with
$$
\frac{1}{2}>t_1>t_2>\ldots >t_d.
$$
Denote $A(q):=\left(\sum_{i=1}^m p_i^q\right)^{1/(q-1)}$. Then by Definitions \eqref{e-1.1}-\eqref{e-1.2}, for $q>0$,
$$
\tau(q)=\left\{
\begin{array}{l}
D(q)=\displaystyle\frac{\log \sum_{i=1}^m p_i^q}{\log t_1} \qquad \mbox { if } A(q)\geq t_1,\\ \mbox{}\\
D(q)=\displaystyle\frac{\log \sum_{i=1}^m p_i^q}{\log
t_{k+1}}+(q-1)\left(k-\frac{\log (t_1\ldots t_k)}{\log t_{k+1}}\right)\\
\qquad \qquad\qquad \mbox { if }  t_1\ldots t_{k+1}\leq A(q)< t_1\ldots t_k   \mbox{ for some }1\leq k\leq d-1,\\
\mbox{}\\

d(q-1)\qquad \mbox{ if } A(q)<t_1\ldots t_d.
\end{array}
\right.
$$

\end{ex}

\bigskip

\begin{rem} We remark that in Example \ref{ex-1.4}, $\tau'(q+)>\tau'(q-)$ at those points $q\in (0,1)$ such that
$A(q)=t_1\ldots t_k$ for some $k\in \{1,2,\ldots,d-1\}$. Indeed, if such   $q$ exists, a direct calculation shows that
$$
\tau'(q+)-\tau'(q-)=\left(\frac{\log \sum_{i=1}^m p_i^q}{q-1}-\Big(\log \sum_{i=1}^m p_i^q\Big)'\right)\cdot \left(\frac{1}{\log t_{k+1}}-\frac{1}{\log t_k}\right)>0,
$$
using the strict convexity of the function $x\mapsto \log \sum_{i=1}^m p_i^x$ on $(0,\infty)$ and $q<1$; therefore $\tau$ is not concave on any neighborhood of $q$.
In this case,  Falconer's formula $\tau(\mu^\ba,t)=\tau(t)$ in Theorem \ref{thm-1.2} can not be extended to all $t\in (0,1)$,  because  $\tau(\mu^\ba,t)$ should be concave over $\R$. A right formula for $\tau(\mu^\ba,t)$ is expected. In Example \ref{e-6.C}, we  provide such a formula for certain  non-overlapping planar IFS.
\end{rem}

The paper is organized  as follows. In section \ref{S-2},  we present some definitions and known results about the sub-additive thermodynamic formalism; we also present
some known dimensional results about the projections of ergodic measures on typical self-affine sets. In section \ref{S-3}, we give a formula for the derivative of $D(q)$ using
the sub-additive thermodynamic formalism.   In section \ref{S-4}, we show that for a class of self-affine IFS on $\R^d$, any associated self-affine measure is either singular or equivalent to  the restricted $d$-dimensional Lebesgue measure
on the attractor.   In section \ref{S-5} we prove Theorem \ref{thm-1.3} and
related results. In section \ref{S-6}, we prove an extension of Falconer's formula for the $L^q$-spectrum and give some complement to Theorem~\ref{thm-1.3}. In section~\ref{S-7} we give further extensions of our results. In Appendix~\ref{S-8} we provide a proof of the concavity of the functions $\tau$ and $D$ over $(1,\infty)$, as well as a proof of their concavity over the subintervals intervals of $(0,1)$ over which $D(q)/(q-1)$ lies between two consecutive integers of $[0,d]$.

\section{Preliminaries}
\label{S-2}

\subsection{The sub-additive thermodynamic formalism}
\label{S-2.1}
In this subsection, we present some definitions and known results about the sub-additive thermodynamic formalism on full shifts.

Let $m\geq 2$.  Let $(\Sigma, \sigma)$ denote the one-sided full shift space over the alphabet $\{1,\ldots, m\}$ (cf. \cite{Bow75}). Let $\M(\Sigma,
\sigma)$ denote the collection of $\sigma$-invariant Borel probability measures on $\Sigma$ endowed with the weak star topology. For $\eta\in \M(\Sigma, \sigma)$, let
$h_\eta(\sigma)$ denote the measure-theoretic entropy of $\eta$ with respect to $\sigma$ (cf. \cite{Bow75}).

A sequence $\Psi=\{\psi_n\}_{n=1}^\infty$ of continuous functions on $\Sigma$ is said to be a {\it sub-additive potential} if
$$
\psi_{n+m}(x)\leq \psi_n(x)+\psi_{m}(\sigma^nx), \qquad \forall\; x\in \Sigma, \; m,n\in \N.
$$
More generally, $\Psi=\{\psi_n\}_{n=1}^\infty$ is said to be an {\it asymptotically sub-additive potential} if for any $\epsilon>0$, there exists a sub-additive potential
$\Phi=\{\phi_n\}_{n=1}^\infty$ on $\Sigma$ such that
$$
\limsup_{n\to \infty} \frac{1}{n}\sup_{x\in \Sigma}|\psi_n(x)-\phi_n(x)|\leq \epsilon.
$$

Now let $\Psi=\{\psi_n\}_{n=1}^\infty$ be an asymptotically sub-additive potential on $\Sigma$. The {\it topological pressure  $P(\sigma, \Psi)$ of $\Psi$} is defined as
$$ P(\sigma, \Psi):=\limsup_{n\to \infty} \frac{1}{n}\log \sum_{I\in \Sigma_n}\sup_{x\in [I]} \exp(\psi_n(x)),
$$
where $\Sigma_n:=\{1,\ldots, m\}^n$ and $[I]=\{x=(x_i)_{i=1}^\infty\in \Sigma: x_1\ldots x_n=I\}$ for $I\in \Sigma_n$.  For $\eta\in \M(\Sigma, \sigma)$, set
$$
\Psi_*(\eta)=\lim_{n\to \infty}\frac{1}{n} \int \psi_n(x)\; d\eta(x).
$$

The following variational principle was proved in \cite{CFH08, FeHu10} in a more general setting.

\begin{pro}
\label{pro-2.1} $P(\sigma, \Psi)=\sup\{ h_\eta(\sigma)+\Psi_*(\eta):\; \eta\in \M(\Sigma, \sigma)\}$.
\end{pro}

We remark that the variational principle for sub-additive potentials has been studied in the literature under additional
assumptions on the corresponding sub-additive potentials (see e.g. \cite{Fal88b, Bar96, FeLa02, Kae04}).

Let $\I(\Psi)$ denote the collection of  $\eta\in \M(\Sigma, \sigma)$ such that
$$h_\eta(\sigma)+\Psi_*(\eta)=P(\sigma, \Psi).$$
Then $\I(\Psi)\neq \emptyset$ (see e.g., \cite[Theorem 3.3]{FeHu10}). Each element of $\I(\Psi)$ is called an {\it equilibrium state} for $\Psi$.

\begin{lem}[\cite{FeHu10}, Theorem 3.3(i)]
\label{lem-2.2}
 $\I(\Psi)$ is a non-empty compact convex subset of $\M(\Sigma, \sigma)$. Moreover, any extreme point of $\I(\Psi)$ is an ergodic measure on $\Sigma$.
\end{lem}
We end this subsection by mentioning the following property of $\Psi_*$; for a proof, see \cite[Proposition A.1(2)]{FeHu10}.
\begin{lem}
\label{lem-2.3}
The map $\Psi_*: \M(\Sigma,\sigma)\to \R\cup\{-\infty\}$ is upper semi-continuous.
\end{lem}

\subsection{Projections of ergodic measures on typical self-affine sets}
In this subsection, we introduce a  result of Jordan, Pollicott and Simon \cite{JPS07}  for  self-affine IFS, which plays a key role in the proof of Theorem
\ref{thm-1.3}.

Let $m\geq 2$ and $T_1,\ldots, T_m$ be non-singular linear transformations  from $\R^d$ to $\R^d$.

For $\ba=(a_1,\ldots, a_m)\in \R^{md}$, let $\pi^\ba:\Sigma\to \R^d$ be the coding mapping associated with the IFS $\{T_i+a_i\}_{i=1}^m$, that is,
\begin{equation}
\label{e-2.21}
\pi^\ba(x)=\lim_{n\to \infty} S_{x_1}\circ S_{x_2}\circ\ldots \circ S_{x_n}({\bf 0}),
\end{equation}
where $S_i:=T_i+a_i$. It is not hard to see that $\pi^\ba(\Sigma)$ is just the attractor of the IFS $\{T_i+a_i\}_{i=1}^m$. For $s\geq 0$ and $\eta\in \M(\Sigma, \sigma)$,
set
\begin{equation}\label{e-2.1}
\phi^s_*(\eta)=\lim_{n\to \infty} \frac{1}{n}\int \log \phi^s(T_{x|n})\; d\eta(x),
\end{equation}
where $T_{x|n}:=T_{x_1}\ldots T_{x_n}$ for $x=(x_i)_{i=1}^n\in \Sigma$ and $\phi^s(\cdot)$ denotes the singular value function (see Definition \ref{de-1.1}). Since
$\phi^s$ is sub-multiplicative in the sense that $\phi^s(AB)\leq \phi^s(A)\phi^s(B)$ for any $d\times d$ real matrices $A,B$ (cf. \cite[Lemma 2.1]{Fal88}), the limit in
\eqref{e-2.1} exists. The following definition was introduced in \cite{JPS07} in a slightly different but equivalent form.

\begin{de}
\label{de-2.2} For an ergodic measure $\eta$ on $\Sigma$, the Lyapunov dimension of $\eta$ (associated with $T_1,\ldots, T_m$), denoted as $ \dim_{LY} \eta$,  is defined by
$\dim_{LY} \eta=s,$
 where $s$ is the unique non-negative value so that $h_\eta(\sigma)+\phi_*^s(\eta)=0$.
\end{de}

Let us give another definition.
\begin{de} Let $\xi$ be a Borel probability measure on $\R^d$.
\begin{itemize}
\item[(i)] The Hausdorff dimension of $\xi$ is defined as
$$
\dim_H\xi=\inf\{\dim_HF:\; F\subset \R^d \mbox{ is Borel with }\xi(\R^d\backslash F)=0\}.
$$
\item[(ii)] Say that $\xi$ is exactly dimensional if there is a constant $c\geq 0$ such that
\begin{equation*}
\lim_{r\to 0}\frac{\log \xi(B(z,r))}{\log r}=c\quad \mbox{ for $\xi$-a.e $z\in \R^d$}.
\end{equation*}
\end{itemize}
\end{de}
It is well known \cite{You82} that if $\xi$ is exactly dimensional, then $\dim_H\xi=c$.  Now we can state the following projection result of Jordan, Pollicott and Simon \cite{JPS07}.

\begin{thm}[\cite{JPS07}]
 \label{thm-2.3}
 Assume that $\|T_i\|<1/2$ for $1\leq i\leq m$. Let $\eta$ be an ergodic measure on $\Sigma$. Then for $\L^{md}$-a.e $\ba\in \R^{md}$,
\begin{itemize}
\item[(i)]  $\dim_H \eta\circ (\pi^{\ba})^{-1}=\min\{\dim_{LY}\eta, d\}$.
\item[(ii)] If $\dim_{LY}\eta\in [0,1]$, then $\eta\circ (\pi^{\ba})^{-1}$ is exactly dimensional.
 \item[(iii)] If $\dim_{LY}\eta>d$, then
    $\eta\circ (\pi^{\ba})^{-1}\ll \L^{d}$.

\end{itemize}

\end{thm}
We remark that Theorem \ref{thm-2.3}(ii) was only implicitly in \cite[Theorem 4.3]{JPS07}.   After we completed the first version of this paper, Thomas Jordan pointed to us that the assumption $\dim_{LY}\eta\in [0,1]$ in Theorem \ref{thm-2.3}(ii) can be removed, that is, for any ergodic measure $\eta$ on $\Sigma$, $\eta\circ (\pi^{\ba})^{-1}$ is exactly dimensional for $\L^{md}$-a.e $\ba\in \R^{md}$; the proof is done  by taking  a minor change in the proof of \cite[Theorem 4.3]{JPS07} for the upper bound \cite{Jor11}.  We remark that this result was proved earlier by  Falconer and Miao \cite{FaMi11}  in the special case that
$\eta$ is a Bernoulli product measure or a Gibbs measure.  However if $T_1,\ldots, T_m$ are commutative, then $\eta\circ (\pi^{\ba})^{-1}$ is exactly dimensional for any ergodic measure $\eta$ on $\Sigma$ and any $\ba\in \R^{md}$ (cf. \cite[Theorem 2.12]{FeHu09}).

\section{A formula for the derivative  of $D(q)$}
\label{S-3}

Assume that $T_1,\ldots, T_m$ are contractive non-singular linear mappings from $\R^d$ to $\R^d$, and let $(p_1,\ldots, p_m)$ be a probability vector. Let $D(q)$ be
defined as in \eqref{e-1.1}. It is not hard to see that for $q>0$, $q\neq 1$,  $D(q)$ is the unique value $s\in \R$ so that
\begin{equation}
\label{e-3.1} \lim_{n\to \infty} \frac{1}{n}\log \sum_{I\in \Sigma_n} \phi^{s/(q-1)}(T_I)^{1-q}p_I^q=0.
\end{equation}

Define $f\in C(\Sigma)$ by
$$
f(x)=\log p_{x_1} \mbox{ for } x=(x_i)_{i=1}^\infty\in \Sigma.
$$
For $q>0$, $q\neq 1$, assume that
\begin{equation}
\label{e-3.2}
\begin{split}
&\left\{ (1-q) \log \phi^{D(q)/(q-1)}(T_{x|n})\right\}_{n=1}^\infty \mbox{ is an asymptotically }\\ &\quad\mbox{sub-additive potential on $\Sigma$. }
\end{split}
\end{equation}
Then by \eqref{e-3.1},  $D(q)$ satisfies the following equation
\begin{equation}\label{e-3.3}
P(\sigma, G_q)=0,
\end{equation}
where $P$ denotes the pressure function (see section \ref{S-2}),  $G_q:=\{g_{n,q}\}_{n=1}^\infty$ is a potential  defined by
\begin{equation}
\label{ee-1}
g_{n,q}(x)=(1-q) \log \phi^{D(q)/(q-1)}(T_{x|n}) +q \sum_{k=0}^{n-1} f(\sigma^k x),
\end{equation}
By the assumption \eqref{e-3.2}, $G_q$ is asymptotically sub-additive.

\begin{rem}
\begin{itemize}
\item[(i)] The assumption \eqref{e-3.2}  always holds when $0<q<1$, since $\phi^s$ is sub-multiplicative for any $s\geq 0$ in the sense that $\phi^s(AB)\leq
    \phi^s(A)\phi^s(B)$ (cf. \cite{Fal88}). \item[(ii)] When $q>1$,  \eqref{e-3.2} holds  if $T_1,\ldots, T_m$ satisfy some additional assumption, for instance, all
    $T_i$ are the same, or each $T_i$ is of  the form   $$T_i=\mbox{diag}(t_{i,1}, t_{i,2},\ldots, t_{i,d}) \mbox{ with }
    t_{i,1}>t_{i,2}>\ldots>t_{i,d}>0.$$
\end{itemize}
\end{rem}
By \eqref{e-3.3} and Proposition \ref{pro-2.1}, we have

\begin{lem}
\label{lem-3.2} Let $q>0$, $q\neq 1$. Assume that \eqref{e-3.2} holds. Then
\begin{equation*}
\begin{split}
0=\sup\left\{h_\eta(\sigma)+(1-q) \phi_*^{D(q)/(q-1)}(\eta)+q \int fd\eta:\; \eta\in \M(\Sigma,\sigma)\right\}.
\end{split}
\end{equation*}
where $\phi^s_*(\cdot)$ is defined as in \eqref{e-2.1}. Moreover,
$$
h_\eta(\sigma)+(1-q) \phi_*^{D(q)/(q-1)}(\eta)+q \int fd\eta=0, \quad \forall\; \eta\in \I(G_q),
$$
where $\I(G_q)$ denotes the collection of the equilibrium states of the potential $G_q$ (cf. Section \ref{S-2.1}).

\end{lem}

For $\eta\in \M(\Sigma,\sigma)$, denote
\begin{equation}
\label{e-3.4} \lambda_i(\eta):=\lim_{n\to \infty}\frac{1}{n}\int \log \alpha_i(T_{x|n})\; d\eta(x), \quad i=1,\ldots, d,
\end{equation}
where $\alpha_i(A)$  denotes the $i$-th singular value of $A$. We write $\lambda_0(\eta)=0$ for convention. It is easy to see that
$\lambda_i(\eta)=\phi^{i}_*(\eta)-\phi^{i-1}_*(\eta)$ for $1\leq i\leq d$. In particular, if $s\in [k, k+1)$ for some integer $0\leq k\leq d-1$, then
\begin{equation}
\label{e-3.6} \phi^s_*(\eta)=\lambda_1(\eta)+\ldots+\lambda_k(\eta)+(s-k) \lambda_{k+1}(\eta)=\phi^k_*(\eta)+(s-k) \lambda_{k+1}(\eta).
\end{equation}

\begin{lem}
\label{lem-915} Let  $\eta$ be an ergodic measure on $\Sigma$.  Then for $\eta$-a.e $x\in \Sigma$,
$$
\lim_{n\to \infty }\frac{\log \alpha_i(T_{x|n})}{n}=\lambda_i(\eta),\qquad i=1,\ldots, d.
$$
\end{lem}
\begin{proof}
Let $s\geq 0$. Since $\phi^s$ is sub-multiplicative, by Kingman's sub-additive ergodic theorem (cf. \cite[Theorem 10.1]{Wal82}),
$$
\lim_{n\to \infty }\frac{\log \phi^s(T_{x|n})}{n}=\phi^s_*(\eta)\quad \mbox{ for $\eta$-a.e. $x\in \Sigma$}.
$$
Now Lemma \ref{lem-915} follows from the fact that $\log \alpha_i(A)=\log \phi^i(A)-\log \phi^{i-1}(A)$ for $i=1,\ldots, d$.
\end{proof}

In the following proposition, we give a formula for the derivative of $D(q)$.
\begin{pro}
\label{pro-3.3} Let $q>0$, $q\neq 1$. Assume that \eqref{e-3.2} holds. If $\frac{D(q)}{q-1}\in (k, k+1)$ for some integer $0\leq k\leq d-1$, then
\begin{equation}
\label{e-3.5}
\begin{split}
&D'(q-)\geq \sup_{\eta\in \I(G_q)}\frac{\int f d\eta-\phi^k_*(\eta)}{\lambda_{k+1}(\eta)}+k,\\ &D'(q+)\leq \inf_{\eta\in \I(G_q)}\frac{\int f
d\eta-\phi^k_*(\eta)}{\lambda_{k+1}(\eta)}+k.\\
\end{split}
\end{equation}
In particular, if in addition $D'(q)$ exists, then
\begin{equation}
\label{e-3.8} D'(q)=\frac{\int f d\eta-\phi^k_*(\eta)}{\lambda_{k+1}(\eta)}+k,\qquad\forall\; \eta\in \I(G_q).
\end{equation}
\end{pro}
\begin{proof}
First fix $\eta\in \I(G_q)$. By \eqref{e-3.6}, we have
$$(1-q)\phi^{D(q)/(q-1)}_*(\eta)=(1-q)(\phi^k_*(\eta)-k\lambda_{k+1}(\eta))-D(q)\lambda_{k+1}(\eta).
$$
Combining this  with Lemma \ref{lem-3.2} yields
\begin{equation}
\label{e-3.7} -D(q)\lambda_{k+1}(\eta)+q A +B=0,
\end{equation}
where $$A:=\int f d\eta-\phi^k_*(\eta)+k\lambda_{k+1}(\eta), \quad B:=h_\eta(\sigma)+\phi^k_*(\eta)-k\lambda_{k+1}(\eta).
$$
For small $\epsilon\in \R$, apply  Lemma \ref{lem-3.2} (in which $q$ is replaced by $q+\epsilon$) to obtain
\begin{equation}
\label{e-3.88} -D(q+\epsilon)\lambda_{k+1}(\eta)+(q+\epsilon) A +B\leq  0.
\end{equation}
Subtracting \eqref{e-3.7} from \eqref{e-3.88} yields
$$
(D(q)-D(q+\epsilon))\lambda_{k+1}(\eta)+\epsilon A\leq 0.
$$
Hence
\begin{equation*}
\begin{split}
&\frac{D(q+\epsilon)-D(q)}{\epsilon}\leq \frac{A}{\lambda_{k+1}(\eta)} \quad\mbox{ if }\epsilon >0, \mbox{ and} \\
& \frac{D(q+\epsilon)-D(q)}{\epsilon}\geq
\frac{A}{\lambda_{k+1}(\eta)} \quad\mbox{ if }\epsilon <0.
\end{split}
\end{equation*}
Letting $\epsilon\to 0$, we obtain
$$
D'(q+)\leq \frac{A}{\lambda_{k+1}(\eta)} \mbox{ and  }D'(q-)\geq \frac{A}{\lambda_{k+1}(\eta)}.
$$
Letting $\eta$ run over $\I(G_q)$, we obtain \eqref{e-3.5}. It implies that if $D'(q)$ exists, then \eqref{e-3.8} holds.
\end{proof}

As the main result of this section, we have
\begin{pro}
\label{pro-3.4} Let $q>0$, $q\neq 1$.
\begin{itemize}
\item[(i)] If $0<q<1$ and $\frac{D(q)}{q-1}\in (0,1)$, then
\begin{equation}
\label{e-3.5'} D'(q-)= \sup_{\eta\in \I(G_q)}\frac{\int f d\eta}{\lambda_{1}(\eta)},\quad D'(q+)= \inf_{\eta\in \I(G_q)}\frac{\int f d\eta}{\lambda_{1}(\eta)}.
\end{equation}
Furthermore, for $\alpha\in \{D'(q+), D'(q-)\}$, there exists an ergodic measure $\eta\in \I(G_q)$ such that $\alpha=\frac{\int f d\eta}{\lambda_{1}(\eta)}$.

\item[(ii)] Assume  that $T_i$ ($i=1,\ldots,m$) are of  the form   \begin{equation}
\label{e-ee1}T_i=\mbox{diag}(t_{i,1}, t_{i,2},\ldots, t_{i,d})
\end{equation}
 with $t_{i,1}>t_{i,2}>\ldots>t_{i,d}>0.$
If $k< \frac{D(q)}{q-1}< k+1$ for some integer $0\leq k\leq d-1$, then  $D'(q)$ exists and there exists an ergodic measure $\eta\in \I(G_q)$ such that  then
\begin{equation}
\label{e-3.8'} D'(q)=\frac{\int f d\eta-\phi^k_*(\eta)}{\lambda_{k+1}(\eta)}+k.
\end{equation}
\end{itemize}
\end{pro}
\begin{proof} We first prove (i). Assume that $0<q<1$ satisfying that $D(q)/(q-1)\in (0,1)$. By continuity, there exists  a neighborhood $\Delta$ of $q$ so that $\Delta\subset (0,1)$ and  $D(t)/(t-1)\in (0,1)$ for any $t\in \Delta$. Let $(q_n)\subset \Delta$ be a sequence so that
$\lim_{n\to \infty}q_n=q$. Take $\eta_n\in \I(G_{q_n})$. By \eqref{e-3.3}, $(G_{q_n})_*(\eta_n)+h_{\eta_n}(\sigma)=0$.
Taking
a subsequence if necessary we may assume that $\eta_n$ converges to some $\eta\in \M(\Sigma,\sigma)$ in the weak-star topology. We claim that
$\eta\in \I(G_q)$ and $\limsup_{n\to \infty}\lambda_1(\eta_n)=\lambda_1(\eta)$.

To prove the claim, we notice that the map $\mu\mapsto \lambda_1(\mu)$ is upper semi-continuous on $\M(\Sigma, \sigma)$. This follows from Lemma \ref{lem-2.3}, in which we take $\Psi=\{\log \phi^1(T_{x|n})\}_{n=1}^\infty$.  For $t\in \Delta$ and $\mu\in \M(\Sigma,\sigma)$, by \eqref{ee-1}, we have
$$
(G_t)_*(\mu)=-D(t)\lambda_1(\mu)+t\int f d\mu.
$$
Hence
\begin{eqnarray*}\limsup_{n\to \infty} (G_{q_n})_*(\eta_n)&=&-D(q)\limsup_{n\to \infty}\lambda_1(\eta_n)+q \int fd\eta\\
&\leq& -D(q)\lambda_1(\eta)+q \int fd\eta=(G_q)_*(\eta).
\end{eqnarray*}
Meanwhile
$\limsup_{n\to \infty} h_{\eta_n}(\sigma)\leq h_\eta(\sigma)$ by the upper semi-continuity of $h_{(\cdot)}(\sigma)$.
It follows that $$(G_q)_*(\eta)+h_\eta(\sigma)\geq \limsup_{n\to \infty} ((G_{q_n})_*(\eta_n)+h_{\eta_n}(\sigma))=0.$$
However, by Proposition \ref{pro-2.1} and \eqref{e-3.3},
$0=P(\sigma,G_q)\geq (G_q)_*(\eta)+h_\eta(\sigma)$. Hence we have $(G_q)_*(\eta)+h_\eta(\sigma)=0=(G_{q_n})_*(\eta_n)+h_{\eta_n}(\sigma)$. Thus $\eta\in \I(G_q)$, and moreover,
$\limsup_{n\to \infty}\lambda_1(\eta_n)=\lambda_1(\eta)$.

Since $D$ is concave in a neighborhood of $q$ (see Proposition \ref{Dconcavity}), we can take two sequences $(s_n)$, $(t_n)$ such that $s_n\uparrow q$,  $t_n\downarrow q$ and $D'(s_n)$, $D'(t_n)$ exist.  Then
$D'(q-)=\lim_{n\to \infty}D'(s_n)$ and $D'(q+)=\lim_{n\to \infty}D'(t_n)$.
Take $\eta_n'\in \I(G_{s_n})$. Taking
a subsequence if necessary, we may assume that $\eta_n'$ converges to some $\eta\in \M(\Sigma,\sigma)$ in the weak-star topology. By the above claim, we have
$\eta\in \I(G_q)$ and $\limsup_{n\to \infty}\lambda_1(\eta_n')=\lambda_1(\eta)$. Hence by Proposition \ref{pro-3.3},
$$
D'(q-)=\lim_{n\to \infty} D'(s_n)=\lim_{n\to \infty}\frac{\int f\; d\eta_n'}{\lambda_1(\eta_n')}=\frac{\int f\;d\eta}{\lambda_1(\eta)}.
$$
Combining this with \eqref{e-3.5} yields $D'(q-)=\sup_{\mu\in\I(G_q)}\frac{\int f\;d\mu}{\lambda_1(\mu)}$. Similarly we can show that $D'(q+)=\inf_{\mu\in\I(G_q)}\frac{\int f\;d\mu}{\lambda_1(\mu)}$.

Now let $\alpha\in \{D'(q-),D'(q+)\}$. Define $$\I_\alpha=\left\{\mu\in \I(G_q):\; \frac{\int f\;d\mu}{\lambda_1(\mu)}=\alpha\right\}.$$
The arguments in the last paragraph imply that $\I_\alpha\neq \emptyset.$ Furthermore one can check that  $\I_\alpha$ is compact and convex. We are going to show that
$\I_\alpha$ contains at least one ergodic measure. Without loss of generality, we assume that $\alpha=D'(q-)$.  By the Krein-Milman theorem
(c.f. \cite[p. 146]{Con-book}), $\I_\alpha$ contains at least one
extreme point, denoted by $\nu$. Let $\nu=p\nu_1+(1-p)\nu_2$ for
some $0<p<1$ and $\nu_1,\nu_2\in \mathcal{M}(\Sigma,\sigma)$. Then
\begin{align*}
P(\sigma, G_q)=h_\nu(\sigma)+(G_q)_*(\nu)
=p(h_{\nu_1}(\sigma)+(G_q)_*(\nu_1))+(1-p)(h_{\nu_2}(\sigma)+(G_q)_*(\nu_2)).
\end{align*}
By Proposition \ref{pro-2.1}, $\nu_1,\nu_2\in
\mathcal{I}(G_q)$. Since
$$\alpha=\sup_{\eta\in \I(G_q)}\frac{\int f d\eta}{\lambda_1(\eta)}=\frac{\int f d\nu}{\lambda_1(\nu)}=\frac{p\int f d\nu_1+(1-p)\int f d\nu_2}{p\lambda_1(\nu_1)+(1-p)\lambda_1(\nu_2)},
$$
we must have $\nu_1,\nu_2\in \I_\alpha$. Since $\nu$ is an extreme point of
$\I_\alpha$, we have $\nu_1=\nu_2=\nu$. It follows that $\nu$ is an extreme point
of $\mathcal{M}(\Sigma,\sigma)$, i.e., $\nu$ is ergodic. Therefore
$\I_\alpha$ contains an ergodic measure. This finishes the proof of (i).

Now we turn to the proof of (ii).   Under the additional assumption \eqref{e-ee1} on $T_i$'s, we can adapt the proof of (i) to show that  if $D(q)/(q-1)\in (k, k+1)$ for some $0\leq k\leq d-1$,
then
\begin{equation}
\label{e-ee}
D'(q-)=\sup_{\eta\in \I(G_q)}\frac{\int f d\eta-\phi^k_*(\eta)}{\lambda_{k+1}(\eta)}+k,\quad D'(q+)=\inf_{\eta\in \I(G_q)}\frac{\int f d\eta-\phi^k_*(\eta)}{\lambda_{k+1}(\eta)}+k.
\end{equation}
Indeed, under this new assumption on $T_i$'s, we see that the potential $G_q=\{g_{n,q}\}$ is additive in the sense that $g_{n,q}=\sum_{i=0}^{n-1}h(x)$ for some continuous function $h$ on $\Sigma$. Moreover, $h(x)$ depends only on the first coordinate of $x$.  Therefore the maps $\mu\mapsto \lambda_k(\mu)$, $\mu\mapsto \phi^k_*(\mu)$ are continuous over $\M(\Sigma,\sigma)$. Based on this fact, \eqref{e-ee} can be proved in a way similar to that of (i). We ignore the details.  Since $h(x)$ only depends on the first coordinate of $x$, $h$ is H\"{o}lder continuous. Therefore $\I(G_q)$ is a singleton consisting of an ergodic measure (see, e.g., \cite[Theorem 1.2]{Bow75}).
This together with \eqref{e-ee} proves \eqref{e-3.8'}.
  \end{proof}

\begin{rem} Assume that $T_i$, $i=1,\ldots,m$, satisfy the following irreducibility condition: there is no proper subspace $V\neq \{\bf 0\}$ of $\R^d$ so that $T_i(V)\subset V$. Then $\phi^1$ satisfies certain quasi-multiplicative property  which guarantees that $\I(G_q)$ is a singleton (and hence $D'(q)$ exists by Proposition \ref{pro-3.4}(i)) provided that
$0<q<1$ and $\frac{D(q)}{q-1}\in (0,1)$.   More generally, when $0<q<1$ and $\frac{D(q)}{q-1}\in (k,k+1)$, $D'(q)$ exists if $T_i$, $i=1,\ldots,m$ satisfy the so-called $C(k+1)$ condition introduced in \cite{FaSl09}. This can be proved in a way similar to \cite[Proposition 1.2]{FeKa11}, or by simply using \cite[Theorem 5.5]{Fen11'}.
\end{rem}

\section{Equivalence of certain self-affine measures to the Lebesgue measure}
\label{S-4}

Our multifractal analysis will need the first part of the following Proposition~\ref{pro-4.1}, which deals with the comparison between the  Lebesgue measure and projections of certain ergodic measures on attractors of self-affine IFS with positive Lebesgue measure; we do not only consider  Bernoulli products measures because our main results extend to Gibbs measures (see Section~\ref{S-7}). The first case considered in Proposition~\ref{pro-4.1} is essentially a restatement of  a result obtained by Shmerkin in \cite[Proposition 22(3)]{Shm06}, while the second one is a nontrivial improvement of \cite[Proposition 22(3)]{Shm06}, in which only  the case $d\le 2$ was treated. In fact in Proposition 22 of \cite{Shm06} Shmerkin only considered self-affine measures, but he mentioned as a remark  that his results are valid for the class of ergodic measures we consider. Though the second case considered in Proposition~\ref{pro-4.1} will not be used in this paper, we think it is worth keeping it in this paper due to the importance of such results in the general ergodic theory of self-affine IFS, and also because the method differs from that used by Shmerkin, by avoiding to refer to general results on density bases. We will also use this approach to give an alternative proof of the first case of Proposition~\ref{pro-4.1} when $d\le 2$.

Let $\{S_i=T_i+a_i\}_{i=1}^m$ be an affine IFS on $\R^d$ with the attractor $F$. Assume that $\L^d(F)>0$. Let $\L^d_F$ denote the restriction of $\L^d$ on $F$, i.e.,
$\L^d_F(A)=\L^d(A\cap F)$ for any Borel set $A\subset \R^d$.

Let $\pi=\pi^\ba:\Sigma\to \R^d$ be defined as in \eqref{e-2.21}.  Let $\eta\in \M(\Sigma,\sigma)$ and  $\mu=\eta\circ \pi^{-1}$.  Say that $\L^d_F$ is {\it equivalent} to $\mu$ if for any Borel set $A\subset \R^d$,  $\L^d_F(A)=0$ if and only if $\mu(A)=0$.

\begin{pro}
\label{pro-4.1} Assume that one of the following conditions fulfills:
\begin{itemize}
\item[(i)] The $T_i$ are diagonal;
\item[(ii)]  $T_1=\ldots=T_m$.
\end{itemize}
 Assume that $\eta$ is ergodic satisfying
 $$
 \eta(B)>0\Longrightarrow \eta (iB)>0 \mbox{ for all }1\leq i\leq m
 $$
 for any Borel set $B\subset \Sigma$, where $iB:=\sigma^{-1}(B) \cap [i]$.
  Then $\mu$ is either singular to $\L^d_F$, or equivalent to $\L^d_F$.
\end{pro}

Our approach to Proposition~\ref{pro-4.1} extends some ideas used in \cite{MS98}, where Mauldin and Simon \cite{MS98} established the first results of this kind for linear IFS and Bernoulli product measures  on $\R$.

First we introduce some notation. Suppose $R$ is a rectangle in $\R^d$ parallel to the axes, i.e. $R$ has the form
$$
R=\prod_{i=1}^d [x_i-a_i,x_i+a_i],\quad\text{where } a_i>0.
$$
For $t>0$, we denote
$$
t R= \prod_{i=1}^d [x_i-ta_i,x_i+ta_i].
$$
Also we denote
$$
\|R\|=\max_{1\le i\le d}a_i.
$$

\begin{lem}\label{lem-4.2}
Suppose $\{R_i\}_{i\in\mathcal F}$ is a countable  family of rectangles in $\R^1$ or $\R^2$ with edges parallel to the axes. Assume that
$\displaystyle\frac{\sup_j\|R_j\|}{\inf_j\|R_j \|}<\infty$.
Then there exists a partition $\{\mathcal{F}_1,\mathcal{F}_2\}$ of $\mathcal F$ such that for  $i=1,2$, there exists  $\widetilde{\mathcal F}_i\subset \mathcal F_i$ satisfying
that
\begin{eqnarray*}
&&R_j\ (j\in \widetilde{\mathcal F}_i)\text{ are disjoint, and }  \bigcup_{j\in \widetilde{\mathcal F}_i}M R_j\supset \bigcup_{j\in \mathcal {F}_i} R_j,
\end{eqnarray*}
where $M=\displaystyle 3\cdot \frac{\sup_j\|R_j\|}{\inf_j\|R_j \|}$.
\end{lem}

\begin{proof}
We only treat the case $d=2$. For convenience, for each rectangle $R$ (with edges parallel to the axes), we use $a_i(R)$, $i=1,2$, to denote the length of the semi-axes of $R$ along the $x_i$
direction.

Partition $\mathcal F$ into
\begin{eqnarray*}
\mathcal{F}_1&=&\{j\in\mathcal{F}: a_1(R_j)=\|R_j\|\} \text{ and }  \\ \mathcal{F}_2&=&\mathcal{F}\setminus \mathcal{F}_1=\{j\in\mathcal{F}: a_1(R_j)<a_2(R_j)=\|R_j\|\}.
\end{eqnarray*}
Without loss of generality we prove the result for the case $i=1$. For $j\in\F_1$, denote $\F_1(j)= \{j'\in\F_1:R_j\cap R_{j'}\neq\emptyset\}$. Also denote $a=
\sup_{j\in\F_1}a_2(R_{j})$.

Choose $\F_1^1$ a maximal family in $\F_1$ such that the rectangles $R_j$, $j\in\F_1^1$ are disjoint, and for each $j\in\F_1^1$ we have $a/2< a_2(R_j)\le a$. By
construction, for each $j_1\in \F_1^1$ we have
$$
M R(j_1)\supset \bigcup_{j\in\F_1(j_1)}R_j,
$$
so
$$
\bigcup_{j_1\in \F_1^1}M R(j_1) \supset \bigcup_{j_1\in \F_1^1}\bigcup_{j\in\F_1(j_1)}R_j\supset \bigcup_{j\in\F_1:\; a/2<a_2(R_j)\le a} R_j,
$$
the last inclusion follows from the maximality of $\F_1^1$.

Suppose that for $k\ge 1$ we have built a subfamily $\F_1^k$ of $\F_1$ such that the rectangles $R_j$, $j\in \F_1^k$, are disjoint and
\begin{equation}\label{covering}
\bigcup_{j_k\in \F_1^k}M R(j_k) \supset   \bigcup_{j_k\in \F_1^k}\bigcup_{j\in\F_1(j_k)}R_j\supset \bigcup_{j\in\F_1: a/2^k<a_2(R_j)\le a} R_j.
\end{equation}
If there is no $j\in\F_1$ such that $a_2(R_j)\le a/2^{k}$ or $\bigcup_{j_k\in \F_1^k}\F_1(j_k)=\F_1$, we set $\F_1^{k+1}=\F_1^k$. Otherwise, let $\F''_1$ be a maximal subfamily
of $\F'_1=\F_1\setminus \bigcup_{j_k\in \F_1^k}\F_1(j_k)$ of disjoint rectangles $R_j$ for which $a'/2<a_2(R_j)\le a'$, where $a'=\sup_{j'\in \F'_1}a_2(R_{j'})\le a/2^{k}$.
Then setting $\F_1^{k+1}=\F_1^k\cup F_1''$ we have
$$
\bigcup_{j_k\in \F_1^{k+1}}M R(j_{k+1}) \supset   \bigcup_{j_{k+1}\in \F_1^{k+1}}\bigcup_{j\in\F_1(j_{k+1})}R_j\supset \bigcup_{j\in\F_1: a/2^{k+1}<a_2(R_j)\le a} R_j.
$$
This yields by induction a non-decreasing sequence of  subfamilies $\F_1^k$ of  $\F_1$ such that the $R_j$, $j\in \F_1^k$, are disjoint and satisfy \eqref{covering}.
Consequently $\widetilde\F_1=\bigcup_{k\ge 1}\F_1^k$ is suitable.
\end{proof}

\begin{lem}\label{lem-4.3}
Let $C$ be a cube  in $\R^d$. Let $\{T_j\}_{j\in \F}$ be a countable family of affine mappings from $\R^d$ to itself, with the same linear part $T$. Then there exists  $\widetilde \F\subset \F$ such that
\begin{eqnarray*}
&&T_j(C) \ (j\in\widetilde\F) \text{ are disjoint, and }\bigcup_{j\in \widetilde \F} T_j (2 C)\supset \bigcup_{j\in \F}T_j (C).
\end{eqnarray*}
\end{lem}
\begin{proof}
It is easy to see that if $T_i(C)\cap T_j(C)\neq\emptyset$ then $T_i(2C)\supset T_j(C)$. Taking $\widetilde\F$, a maximal subfamily of
$\F$ such that the parallelepipeds $T_i(C)$, $i\in \widetilde\F$, are pairwise disjoint, we are done.
\end{proof}

\begin{proof}[{\bf Proof of Proposition~\ref{pro-4.1}}] (Case $\mathrm{(i)}$ with $d\le 2$ and case $\mathrm{(ii)}$ in general). We first show that  $\mu$ is either singular or absolutely continuous with respect to $\L^d_F$ (this actually holds for all IFS rather than affine IFS).  This fact is  known  when $\eta$ is a Bernoulli product measure \cite{JeWi35,Hut81}.  Now we consider the general case that  $\eta$ is an ergodic measure. Assume that $\mu$ is not absolutely continuous with respect to $\L^d_F$. Then there is
a Borel set $A\subset F$ such that $\L^d(A)=0$ but $\mu(A)>0$.  Define $W=\pi^{-1}(A)$. Then $\eta(W)=\mu(A)>0$. Since $\eta$ is ergodic,
we have
$
\eta\left(\bigcup_{n=1}^\infty\sigma^{-n}W\right)=1
$ (cf. \cite[Theorem 1.5(iii)]{Wal82}).
Denote $\widetilde{W}:=\bigcup_{n=1}^\infty\sigma^{-n}W$. Then
$$\pi(\widetilde{W})=\bigcup_{n=1}^\infty\bigcup_{1\leq i_1,\ldots,i_n\leq m} S_{i_1\ldots i_n}(A).
$$
Since $S_1,\ldots, S_m$ are contractive, we have $\L^d(S_{i_1\ldots i_n}(A))\leq \L^d(A)=0$, and thus $\L^d(\pi(\widetilde{W}))=0$. However, $\mu(\pi(\widetilde{W}))=\eta\circ \pi^{-1}(\pi(\widetilde{W}))\geq \eta(\widetilde{W})=1$. Hence $\mu$ is singular with respect to $\L^d_F$. Up to now we have shown the claim that $\mu$ is either singular or absolutely continuous with respect to $\L^d_F$.

Assume that the conclusion of Proposition~\ref{pro-4.1} does not hold. Then $\mu$ is absolutely continuous with respect
to $\L^d_F$, but $\L^d_F$ is not absolutely continuous with respect to $\mu$. Hence there exists a Borel set $A\subset F$ with $\L^d_F(A)>0$, but $\mu(A)=0$.

Note that $\mu$ satisfies the following relation for all $k\ge 1$:
$$
\mu(A)=\eta\circ \pi^{-1}(A)=\sum_{1\le i_1,i_2,\ldots, i_k\le m}\eta ([i_1\cdots i_k]\cap \sigma^{-k}\pi^{-1}(S_{i_1\ldots i_k}^{-1}(A)),
$$
from which we obtain that for any $1\le i_1,i_2,\ldots, i_k\le m$,  $$\eta ([i_1\cdots i_k]\cap \sigma^{-k}\pi^{-1}(S_{i_1\ldots i_k}^{-1}(A))=0$$ and
thus $\eta(\pi^{-1}(S_{i_1\ldots i_k}^{-1}(A)))=0$ (by the assumption on $\eta$). Hence
$$
\mu(S_{i_1\ldots i_k}^{-1}(A))=0,\quad
$$
Denote
$$
\Lambda=\Big (\bigcup_{k=1}^\infty\bigcup_{1\le i_1,i_2,\ldots, i_k\le m}S_{i_1\ldots i_k}^{-1}(A)\Big )\cup A.
$$
Then $\mu(\Lambda)=0$, but $\L^d_F(\Lambda)>0$.

In the following, we will show that $\L^d_F(F\setminus \Lambda)=0$, which leads to $\mu(F\setminus \Lambda)=0$ (since $\mu\ll\L^d_F$), and thus
$\mu(F)=\mu(\Lambda)+\mu(F\setminus \Lambda)=0$, a contradiction.  Denote $\Lambda^c=F\setminus \Lambda$. Then $S_i(\Lambda^c)\subset\Lambda^c$ for all $1\le i\le m$.

Assume on the contrary that $\L^d_F(\Lambda^c)>0$. %Set $C=\max_{1\le i\le m}\| S_i\|$.

Now we prove the following general fact: if a Borel subset $E$ of $F$ is such that $S_i(E)\subset E$ for all $1\le i\le m$ and $\L^d(E)>0$, then $\L_F^d (F\setminus E)=0$. In the case of $E=\Lambda^c$, this yields $F\setminus \Lambda$  has zero $\L^d$ measure, i.e. $\Lambda$ has zero
$\L^d_F$ measure, contradicting the assumption $\L^d_F(\Lambda)>0$.

For $0<r<1$, define
$$
\A_r=\Big\{i_1i_2\ldots i_k\in\Sigma^*: \|S_{i_1\ldots i_k}\|\le r,\ \|S_{i_1\ldots i_{k-1}}\|>r\Big\},
$$
where $\Sigma^*=\bigcup_{k=0}^\infty\{1,\ldots,m\}^k$.

Without loss of generality, assume that $F$ is contained in the unit cube $\mathcal C=[0,1]^d$ in $\R^d$.

Let $x\in F$. Denote
$$
\A_{r,x}=\{I\in \A_r: B_r(x)\cap S_I(F)\neq\emptyset\}.
$$
Then
$$
B_{2\sqrt{d}r}(x)\supset \bigcup_{I\in\A_{r,x}}S_I(\mathcal C).
$$
Suppose that the assumptions of Proposition~\ref{pro-4.1} are fulfilled. Then by Lemmas~\ref{lem-4.2} (applied in the case $\mathrm{(i)}$ and when $d\le  2$) and~\ref{lem-4.3} (applied in the case $\mathrm{(ii)}$), there exists a constant $M>0$ ($M=3\lambda^{-1}$, where $\lambda$ is the smallest eigenvalue among those of $T_1,\ldots T_m$ in  case $\mathrm{(i)}$, $M=2$ in case  $\mathrm{(ii)}$), a partition
$\{\A_{r,x}^1, A_{r,x}^2\}$ of $\A_{r,x}$, and for $i=1,2$, a subfamily  $\widetilde \A_{r,x}^i$ of $ \A_{r,x}^i$ such that
\begin{eqnarray*}
&&S_I(\mathcal C),\ I\in \widetilde\A_{r,x}^i, \text{ are disjoint and} \\ &&\bigcup_{I\in\widetilde \A_{r,x}^i}MS_I(\mathcal C)\supset \bigcup_{I\in \A_{r,x}^i}S_I(\mathcal C).
\end{eqnarray*}
Therefore,
$$
\sum_{I\in\widetilde \A_{r,x}^i} \L^d(S_I(\mathcal C))\ge \frac{1}{M^d}\L^d\Big (\bigcup_{I\in \A_{r,x}^i}S_I(\mathcal C)\Big),
$$
and (the sets $S_I(E)$, $I\in \widetilde \A_{r,x}^i$, are necessarily pairwise disjoint)
\begin{eqnarray*}
 \L^d\Big (\bigcup_{I\in\A_{r,x}}S_I(E)\Big )\ge \sum_{I\in\widetilde \A_{r,x}^i} \L^d(S_I(E))&\ge &\frac{\L^d(E)}{\L^d(\mathcal C)}\sum_{I\in\widetilde \A_{r,x}^i} \L^d(S_I(\mathcal C))\\ &\ge
&\frac{\L^d_F(E)}{M^d}\L^d\Big (\bigcup_{I\in \A_{r,x}^i}S_I(\mathcal C)\Big).
\end{eqnarray*}
Denote $\widetilde c=\displaystyle \frac{\L^d_F(E)}{M^d}$. Summing the above inequality over $i\in\{1,2\}$ and using the subadditivity of
$\L^d$  we get
\begin{equation}
\label{e-4-1}
2\L^d\Big (\bigcup_{I\in \A_{r,x}}S_I(E)\Big)\ge \widetilde c\, \L^d\Big (\bigcup_{I\in \A_{r,x}}S_I(\mathcal C)\Big)\ge \widetilde c\, \L^d\Big ( \bigcup_{I\in
\A_{r,x}}S_I(F)\Big).
\end{equation}
If $x$ is a Lebesgue density point of $F$, then when $r$ is sufficiently small,
$$
\L^d(B_r(x)\cap F)\ge \frac{1}{2}r^d,
$$
hence
\begin{equation}
\label{e-4-2}
\L^d\Big ( \bigcup_{I\in \A_{r,x}}S_I(F)\Big)\ge \L^d\Big (B_r(x)\cap \bigcup_{I\in \A_{r,x}}S_I(F)\Big)\ge \frac{1}{2}r^d.
\end{equation}
Thus
\begin{eqnarray*}
\L^d\Big (B_{2\sqrt{d}r}(x)\cap E\Big )&\ge& \L^d\Big (B_{2\sqrt{d}r}(x)\cap  \bigcup_{I\in \A_{r,x}}S_I(E)\Big )\ge \L^d\Big ( \bigcup_{I\in \A_{r,x}}S_I(E)\Big
)\\
&\ge& \frac{ \widetilde c}{2}\, \L^d\Big ( \bigcup_{I\in
\A_{r,x}}S_I(F)\Big)\qquad (\mbox{ by \ref{e-4-1}})\\
&\ge & \frac{\widetilde c}{4}r^d. \qquad (\mbox{ by \ref{e-4-2}})
\end{eqnarray*}
Consequently, every Lebesgue point of $F$ is a point of density in $E$. This implies that $F\setminus E$ has zero $\L^d$ measure.
\end{proof}

\section{The proof of Theorem \ref{thm-1.3}}
\label{S-5}

First we consider the most general case that $T_1$, \ldots, $T_m$ are non-singular linear mappings from $\R^d$ to $\R^d$ satisfying $\|T_i\|<1/2$ for $1\leq i\leq m$.

The following lemma was proved by Falconer (see  \cite[Theorem 6.2 (a)]{Fal99}).

\begin{lem}[Theorem 6.2 (a) of \cite{Fal99}]
\label{lem-5.1} Let $q>0$, $q\neq 1$. For all $\ba\in \R^{md}$, we have $\tau(\mu^\ba, q)/(q-1)\leq \tau(q)/(q-1)$.
\end{lem}

\begin{de}
For any Borel probability measure $\xi$ on $\R^d$ and $z\in \mbox{supp}(\xi)$, the local upper and lower dimensions of $\xi$ at $z$ are defined respectively by
$$\overline{d}(\xi,z):=\limsup_{r\to 0} \frac{\log \xi(B_r(z))}{\log r},\quad \underline{d}(\xi,z):=\liminf_{r\to 0} \frac{\log \xi(B_r(z))}{\log r}.$$
If $\overline{d}(\xi,z)=\underline{d}(\xi,z)$, we use $d(\xi,z)$ to denote the common value, and call it the local dimension of $\xi$ at $z$.
\end{de}

\begin{lem}
\label{lem-5.2} For any $\beta\in \R$ and $q>0$,
$$
\dim_H \{z\in \R^d:\; \underline{d}(\mu^\ba,z)\leq \beta\} \leq \beta q-\tau(\mu^\ba,q),
$$
where  we take the convention $\dim_H\emptyset =-\infty$.
\end{lem}
\begin{proof}
The lemma actually holds for any compactly supported Borel probability measure on $\R^d$. It can be proved by using a simple box-counting argument. For details, see e.g., Proposition
2.5(iv) in \cite{Ols95}.
\end{proof}

\begin{lem}
\label{lem-5.3} Let $\ba\in \R^{md}$. For any Borel set $A\subset \R^d$ and any $i_1,\ldots, i_n\in \{1,\ldots,m\}$,
$$
\mu^\ba(A)\geq p_{i_1}\ldots p_{i_n} \mu^\ba\left(\left(S_{i_1}\circ \ldots \circ S_{i_n}\right)^{-1}(A)\right).
$$
\end{lem}
\begin{proof}
Iterating the self-similar relation $\mu^\ba=\sum_{i=1}^m p_i \mu^\ba\circ S_i^{-1}$ for $n$ times, we have
$$
\mu^\ba=\sum p_{j_1}\ldots p_{j_n} \mu^\ba\circ \left(S_{j_1}\circ \ldots \circ S_{j_n}\right)^{-1},
$$
where the sum is taken over all tuples $(j_1,\ldots, j_n)\in \{1,\ldots,m\}^n$.  Now Lemma \ref{lem-5.3} follows.
\end{proof}

\begin{pro}
\label{pro-5.1} Let  $T_1$, \ldots, $T_m$ be non-singular linear mappings from $\R^d$ to $\R^d$ satisfying $\|T_i\|<1/2$ for $1\leq i\leq m$. Let $q\in (0,1)$ and
$\alpha\in \{D'(q-), D'(q+)\}$.  Assume that $D(q)/(q-1)<1$ and $\alpha q-D(q)\leq 1$. Then  for  $\L^{md}$-a.e. $\ba\in \R^{md}$, $\tau(\mu^\ba,q)=\tau(q)=D(q)$,
$E(\mu^\ba,\alpha)\neq \emptyset$ and furthermore,
$$\dim_H E(\mu^\ba,\alpha)=\alpha q-\tau(q).$$

\end{pro}

\begin{proof}
 Since $D(q)/(q-1)<1\leq d$, by \eqref{e-1.2}, we have $\tau(q)=D(q)$. Let $\alpha\in \{\tau'(q+), \tau'(q-)\}$. Then by Proposition \ref{pro-3.4}(i), there exists an
 ergodic measure $\eta\in \I(G_q)$ such that
\begin{equation}
\label{e-5.1} \alpha= \frac{\int f \; d\eta}{\lambda_1(\eta)}.
\end{equation}
This together with \eqref{e-3.7} yields $\alpha q-\tau(q)=-\frac{h_\eta(\sigma)}{\lambda_1(\eta)}$.
Since $\alpha q-\tau(q)\leq 1$ by assumption, due to  \eqref{e-5.1} and Definition \ref{de-2.2}, we have
\begin{equation}
\label{e-5.2} \dim_{LY} \eta=\alpha q-\tau(q)\leq 1.
\end{equation}

Take $\ba\in \R^{md}$ so that $\eta\circ (\pi^{\ba})^{-1}$ is exactly dimensional and $\dim_H \eta\circ (\pi^{\ba})^{-1}=\alpha q-\tau(q)$. By Theorem \ref{thm-2.3}, the
set of such points $\ba$ has the full $md$-dimensional Lebesgue measure.  Take a large $R$ so that $B({\bf 0}, R)$ contains the attractor of the IFS  $\{S_i=T_i+a_i\}_{i=1}^m$.
(Here and afterwards, we  also write $B(z,r)$ for $B_r(z)$.) Then for any $x=(x_i)_{i=1}^\infty\in \Sigma$ and $n\in \N$, by Lemma \ref{lem-5.3} we have
\begin{equation}
\label{e-5.3}
\begin{split}
\mu^\ba\left(B\left(\pi^\ba x, 2R\|T_{x|n}\|\right)\right)&\geq p_{x|n} ~\mu^\ba\left( S_{x|n}^{-1}B\left(\pi^\ba x, 2R\|T_{x|n}\|\right)\right)\\ &\geq p_{x|n}\;\mu^\ba
(B({\bf 0},R))=p_{x|n},
\end{split}
\end{equation}
where in the second inequality we have used an easily checked fact  $$S_{x|n}(B({\bf 0}, R))\subset B(\pi^\ba x, 2R\|T_{x|n}\|).$$ By \eqref{e-5.3}, we have
$$
\overline{d}(\mu^\ba, \pi^\ba x)\leq \limsup_{n\to \infty} \frac{\log p_{x|n}}{\log \|T_{x|n}\|},\quad x\in \Sigma.
$$
By Kingman's  sub-additive ergodic theorem and \eqref{e-5.1}, we have
\begin{equation}
\label{e-5.4} \overline{d}(\mu^\ba,\pi^\ba x)\leq \frac{\int f \; d\eta}{\lambda_1(\eta)}=\alpha \;\mbox{ for $\eta$-a.e }x\in \Sigma.
\end{equation}

Take an strictly increasing sequence $(\alpha_n)$ so that $\lim_{n\to \infty}\alpha_n=\alpha$.   Then by Lemmas \ref{lem-5.2}-\ref{lem-5.1}, for each $n$,
\begin{equation}
\label{e-916} \dim_H\{z\in \R^d:\; \underline{d}(\mu^\ba, z)\leq \alpha_n\}\leq \alpha_n q-\tau(\mu^\ba, q)<\alpha q-\tau(q).
\end{equation}
Since $\eta\circ (\pi^{\ba})^{-1}$ is exactly dimensional and $\dim_H \eta\circ (\pi^{\ba})^{-1}=\alpha q-\tau(q)$, we must have
\begin{equation}
\label{e-917} \eta\circ (\pi^{\ba})^{-1}\{z\in \R^d:\; \underline{d}(\mu^\ba, z)\leq \alpha_n\}=0, \quad n=1,2,\ldots;
\end{equation}
for otherwise if the left-hand side of \eqref{e-917} is greater than $0$, then $$\dim_H\{z\in \R^d:\; \underline{d}(\mu^\ba, z)\leq \alpha_n\}\geq \dim_H \eta\circ
(\pi^\ba)^{-1}=\alpha q-\tau(q),$$ which contradicts \eqref{e-916}. Hence
$$
\eta\circ (\pi^{\ba})^{-1}\{z\in \R^d:\; \underline{d}(\mu^\ba, z)< \alpha\}=0.
$$
Equivalently, we have
$$
\eta\{x\in \Sigma:\; \underline{d}(\mu^\ba, \pi^\ba x)< \alpha\}=0.
$$
This combining with \eqref{e-5.4} yields
$$
\eta\{x\in \Sigma:\; d(\mu^\ba, \pi^\ba x)= \alpha\}=1.
$$
Hence
$$
\dim_H\{z\in \R^d:\; d(\mu^\ba, z)=\alpha\}\geq \dim_H\eta\circ (\pi^\ba)^{-1}=\alpha q-\tau(q).
$$
However by Lemma \ref{lem-5.2}, $\alpha q-\tau(\mu^\ba,q)$ is an upper-bound for the left-hand side of the above inequality, therefore we must have $\alpha
q-\tau(\mu^\ba,q)\geq  \alpha q-\tau(q)$. But by  Lemma \ref{lem-5.1}, we have  $\tau(\mu^\ba,q)\geq \tau(q)$ (noting that $q<1$). Thus we have the equalities
$\tau(\mu^\ba,q)=\tau(q)$ and
$$
\dim_H\{z\in \R^d:\; d(\mu^\ba, z)=\alpha\}=\alpha q-\tau(q).
$$
This finishes the proof of Proposition \ref{pro-5.1}.
\end{proof}

In the reminder part of this section, we shall put more assumption on the linear maps $T_i$ ($1\leq i\leq m$).

\begin{pro}
\label{pro-5.2} Assume that $T_i$ ($i=1,\ldots,m$) are of  the form   $$T_i=\mbox{diag}(t_{i,1}, t_{i,2},\ldots, t_{i,d})$$
 with  $\frac{1}{2}>t_{i,1}>t_{i,2}>\ldots>t_{i,d}>0.$
  Let $q\in (1,2)$. Assume that there exists
an integer $k\in \{0,\ldots, d-1\}$ such that
$$
D(q)/(q-1)\in (k, k+1) \mbox{ and } \alpha q-D(q)\in (k, k+1),
$$
where $\alpha=D'(q)$.
\begin{itemize}
 \item If $k=0$,  then
 for  $\L^{md}$-a.e. $\ba\in \R^{md}$,
$E(\mu^\ba,\alpha)\neq \emptyset$ and
$$\dim_H E(\mu^\ba,\alpha)=\alpha q-\tau(q).$$
\item If $k>0$, then
 for  $\L^{md}$-a.e. $\ba\in \R^{md}$,
$\underline{E}(\mu^\ba,\alpha)\neq \emptyset$ and
$$\dim_H \underline{E}(\mu^\ba,\alpha)=\alpha q-\tau(q).$$
\end{itemize}

 \end{pro}

%\begin{rem}
%Since $\tau$ is concave on $(1,\infty)$ and $D(1)=0$, we always have $D'(q) q-D(q)\leq D(q)/(q-1)$ for $q>1$. Hence, when $D(q)/(q-1)<1$ for some $q>1$, we must have %$D'(q) q- D(q)<1$.
%\end{rem}

\begin{proof}[Proof of Proposition \ref{pro-5.2}]
First consider the case that $k=0$. In this case, we can take a proof essentially identical to that of Proposition \ref{pro-5.1}.   The main difference lying here is that
we directly assume that $\tau(\mu^\ba,q)=\tau(q)$ (since $q\in (1,2)$, by Theorem \ref{thm-1.2}, the set of all such $\ba$ has the full $md$-dimensional Lebesgue
measure).

Next we consider the case that $1\leq k\leq d-1$. Let $\mu$ denote the Bernoulli product measure $\prod_{i=1}^\infty \{p_1,\ldots, p_m\}$ on $\Sigma$. Since $q>1$ and
$D(q)/(q-1)>k$, by Lemma \ref{lem-3.2}, we have
\begin{equation*}
\begin{split}
h_\mu(\sigma)+(1-q) \phi_*^k(\mu)+q\int f d\mu<h_\mu(\sigma)+(1-q) \phi_*^{D(q)/(q-1)}(\mu)+q\int f d\mu\leq 0.
\end{split}
\end{equation*}
Hence $h_\mu(\sigma)+(1-q) \phi_*^k(\mu)+q\int f d\mu<0$, thus $h_\mu(\sigma)+\phi^k_*(\mu)>0$ (noting that  $\int f d\mu=-h_\mu(\sigma)$). By Definition \ref{de-2.2}, we
have
\begin{equation}
\label{e-5.5} \dim_{LY}\mu>k.
\end{equation}

 By Proposition \ref{pro-3.4}(ii), there exists an ergodic measure $\eta$ on $\Sigma$ such that
 $$
 \alpha=\frac{\int fd\eta-\phi^k_*(\eta)}{\lambda_{k+1}(\eta)}+k.
 $$
This together with \eqref{e-3.7} yields
 $$
 \alpha q-D(q)=\frac{h_\eta(\sigma)+\phi^k_*(\eta)}{-\lambda_{k+1}(\eta)}+k.
 $$
 Since by assumption $\alpha q-D(q)\in (k, k+1)$, by Definition \ref{de-2.2}, we have
\begin{equation}
\label{e-5.6} \dim_{LY}\eta=\alpha q-D(q)>k.
\end{equation}

 Let $\Xi_k$ be the canonical projection from $\R^d$ to $\R^k$ defined by $(y_1, y_2,\ldots, y_d)\mapsto (y_1,\ldots, y_k)$. For $\ba=(a_1,\ldots, a_m)\in \R^{md}$,
 denote
 $$
 \pi_k^\ba:=  \Xi_k \circ \pi^\ba.
 $$
 It is easy to see that $\pi_k^\ba$ is the coding map associated with the new IFS $\{\widetilde{T_i}+\Xi_k (a_i)\}_{i=1}^m$, where
 $\widetilde{T_i}=\mbox{diag}(t_1,\ldots, t_{k})$. According to \eqref{e-5.5}-\eqref{e-5.6}, we have also $\dim_{LY}\mu>k$, $\dim_{LY}\eta>k$ (associated with
 $\{\widetilde{T_i}\}_{i=1}^m$). Thus by Theorem \ref{thm-2.3}, for $\L^{md}$-a.e $\ba\in \R^{md}$,  both $\eta\circ  (\pi_k^\ba)^{-1}$ and $\mu\circ (\pi_k^\ba)^{-1}$ are
 absolutely continuous to the $k$-dimensional Lebesgue measure, and hence by Proposition \ref{pro-4.1}, $\eta\circ  (\pi_k^\ba)^{-1}\ll\mu\circ (\pi_k^\ba)^{-1}$ (since
 $\mu\circ (\pi_k^\ba)^{-1}$ is equivalent to the restriction of $\L^d$ on $F^\ba$, where $F^\ba=\pi^\ba(\Sigma)$).

 Now fix $\ba=(a_1,\ldots, a_m)\in \R^{md}$ so that  $\eta\circ  (\pi_k^\ba)^{-1}$ and $\mu\circ (\pi_k^\ba)^{-1}$ are equivalent and $\tau(\mu^\ba, q)=\tau(q)$.
 We have the following.

 \begin{lem}
 \label{lem-5.7} Let $\ell=\mbox{diam} F^\ba$, where $F^\ba=\pi^\ba(\Sigma)$. For any $\delta>0$, we have $\eta(A_\delta)=0$, where
 \begin{equation*}
 \begin{split}
 &A_\delta:= \Big\{x\in \Sigma:\;  \mu^\ba(B(\pi^\ba x, \sqrt{d}\ell \alpha_{k+1}(T_{x|n}))\leq p_{x|n} \exp(-n\phi^{k}_*(\eta)+nk\lambda_{k+1}(\eta)-\delta n)\\
  &\mbox{}\qquad \qquad\mbox{ for all large enough $n$}\Big\}.
 \end{split}
 \end{equation*}
 \end{lem}
We will give the proof of the above lemma a little bit later.   Now we use it to complete the proof of Proposition \ref{pro-5.2}. Since $\eta(A_\delta)=0$, we have for
$\eta$-a.e. $x\in X$,
$$
\frac{\log (\mu^\ba(B(\pi^\ba x, \sqrt{d}\ell \alpha_{k+1}(T_{x|n}) )} {\log \alpha_{k+1}(T_{x|n})}\leq \frac{ \log (p_{x|n}
\exp(-n\phi^{k}_*(\eta)+nk\lambda_{k+1}(\eta)-\delta n))} {\log \alpha_{k+1}(T_{x|n})}
$$
for infinitely many  $n$. Then applying Kingman's  sub-additive ergodic theorem and letting $\delta\to 0$, we obtain
$$
\underline{d}(\mu^\ba, \pi^\ba x)\leq \frac{\int f d\eta-\phi^k_*(\eta)}{\lambda_{k+1}(\eta)}+k =\alpha\quad  \mbox{ for $\eta$-a.e $x\in \Sigma$}.
$$
This plays a similar role as \eqref{e-5.4} in Proposition  \ref{pro-5.1}. To complete the proof, we can  use  the same argument as in the proof of Proposition \ref{pro-5.1}
(the only difference lying here is that we already have the equality $\tau(\mu^\ba, q)=\tau(q)$.).
 \end{proof}

\begin{proof}[Proof of Lemma \ref{lem-5.7}]

For $z=(z_1,\ldots,z_d)\in \R^d$ and $t_1,\ldots, t_d>0$, denote
$$
W(z; t_1,\ldots, t_d):=\prod_{i=1}^d[z_i-t_i, z_i+t_i],\quad \widetilde W((z_1,\ldots, z_k); t_1,\ldots, t_k):=\prod_{i=1}^k[z_i-t_i, z_i+t_i].
$$
In particular, for $r>0$, denote $Q_{r}(z):= \prod_{i=1}^d[z_i-r, z_i+r]$. It is clear that
\begin{equation}
\label{e-z1} Q_r(z)\subset B(z, \sqrt{d}r),\quad\forall\; z\in \R^d, \; r>0.
\end{equation}

Now fix $\delta>0$. Denote
\begin{equation*}
 \begin{split}
 &A':= \Big\{x\in \Sigma:\;  \mu^\ba(Q_{\ell \alpha_{k+1}(T_{x|n})}(\pi^\ba x))\leq p_{x|n} \exp\left(n(1+\delta)(k\lambda_{k+1}(\eta)-\phi^{k}_*(\eta))\right)\\
  &\mbox{}\qquad \qquad\mbox{ for large enough $n$}\Big\}.
 \end{split}
 \end{equation*}
By \eqref{e-z1}, we have $A_\delta\subset A'$. Hence to show $\eta(A_\delta)=0$, it suffices to show that $\eta(A')=0$.

Notice that  for any $x\in \Sigma$ and $n\in \N$,
\begin{equation*}
\begin{split}
&\mbox{} S_{x|n}^{-1} (Q_ {\ell\alpha_{k+1}(T_{x|n})}(\pi^\ba x))\\ &=W\left(\pi^\ba \sigma^nx;
\frac{\ell\alpha_{k+1}(T_{x|n})}{\alpha_{1}(T_{x|n})},\frac{\ell\alpha_{k+1}(T_{x|n})}{\alpha_{2}(T_{x|n})},\ldots,\frac{\ell\alpha_{k+1}(T_{x|n})}{\alpha_{d}(T_{x|n})}
\right)\\ &\supset W\left(\pi^\ba \sigma^nx;
\frac{\ell\alpha_{k+1}(T_{x|n})}{\alpha_{1}(T_{x|n})},\frac{\ell\alpha_{k+1}(T_{x|n})}{\alpha_{2}(T_{x|n})},\ldots,\frac{\ell\alpha_{k+1}(T_{x|n})}{\alpha_{k}(T_{x|n})},
\ell, \ldots,\ell \right).\\
\end{split}
\end{equation*}
It follows that
\begin{equation}\label{e-z2}
\begin{split}
&\mbox{}\mu^\ba\left(Q_ {\ell\alpha_{k+1}(T_{x|n})}(\pi^\ba x))\right)\\ &\geq p_{x|n}\;\mu^\ba\left( S_{x|n}^{-1} (Q_ {\ell\alpha_{k+1}(T_{x|n})}(\pi^\ba x))\right)\qquad
\qquad (\mbox{by Lemma \ref{lem-5.3}})\\ &\geq p_{x|n}\; \mu^\ba\left( W\left(\pi^\ba \sigma^nx;
\frac{\ell\alpha_{k+1}(T_{x|n})}{\alpha_{1}(T_{x|n})},\frac{\ell\alpha_{k+1}(T_{x|n})}{\alpha_{2}(T_{x|n})},\ldots,\frac{\ell\alpha_{k+1}(T_{x|n})}{\alpha_{k}(T_{x|n})},
\ell, \ldots,\ell \right)\right).\\ &= p_{x|n}\; \mu^\ba_k \left( \widetilde W\left(\pi^\ba_k \sigma^nx;
\frac{\ell\alpha_{k+1}(T_{x|n})}{\alpha_{1}(T_{x|n})},\frac{\ell\alpha_{k+1}(T_{x|n})}{\alpha_{2}(T_{x|n})},\ldots,\frac{\ell\alpha_{k+1}(T_{x|n})}{\alpha_{k}(T_{x|n})}
\right)\right),\\
\end{split}
\end{equation}
here we write for brevity $\mu^\ba_k:=\mu\circ (\pi^\ba_k)^{-1}$.

 For $n\in \N$, let $\Omega_n$ denote the set of $x\in \Sigma$ such that
$$
\frac{\ell\alpha_{k+1}(T_{x|n})}{\alpha_{i}(T_{x|n})}\geq \exp(j(1+\delta/2)(\lambda_{k+1}(\eta)-\lambda_{i}(\eta))),\quad  \forall\; i=1,\ldots, k.
$$
Then by Lemma \ref{lem-915}, $\lim_{n\to \infty} \eta\left(\bigcap_{j=n}^\infty \Omega_j\right)=1$.

 Furthermore denote
\begin{equation*}
\begin{split}
&A'_n= \left\{x\in \Sigma:\;  \mu^\ba\left(B(\pi^\ba x, \sqrt{d}\ell \alpha_{k+1}(T_{x|n})\right)\leq p_{x|n}
\exp\left(n(1+\delta)(k\lambda_{k+1}(\eta)-\phi^{k}_*(\eta))\right)\right\},\\ &\qquad u_{n,i}=\exp\left(n(1+\delta/2)(\lambda_{k+1}(\eta)-\lambda_{i}(\eta))\right),\qquad
i=1,\ldots, k.\\ &C_n= \left\{x\in \Sigma:\;  \mu^\ba_k(\widetilde{W}(\pi^\ba_k\sigma^nx; u_{n,1},\ldots, u_{n,k}))\leq
\exp(n(1+\delta)(k\lambda_{k+1}(\eta)-\phi^{k}_*(\eta)))\right\}.
\end{split}
\end{equation*}
By \eqref{e-z2}, we have $A_n'\cap \Omega_n\subset C_n$.

To complete our proof, we need some further notation. For $n\in \N$, denote
$$
{\mathcal R}_n:=\left\{\prod_{i=1}^k [h_i u_{n,i}/2, (h_i+1)u_{n,i}/2):\; h_1,\ldots, h_k\in \Z\right\}.
$$
Clearly, ${\mathcal R}_n$ is a partition of $\R^k$ by rectangles of edge lengths $u_{n,1}$, \ldots, $u_{n,k}$. For any $w\in \R^k$, let $R_n(w)$ denote the element in
${\mathcal R}_n$ that contains $w$.
 Notice that for any $R\in  {\mathcal R}_n$,
$$
\L^k(R)=\prod_{i=1}^k u_{n,i}=\exp(n(1+\delta/2)(k\lambda_{k+1}(\eta)-\phi^{k}_*(\eta))).
$$
It follows that  if $w\in  \pi^\ba_k(\Sigma)$ satisfies
$$
\mu^\ba_k(\widetilde{W}(w; u_{n,1},\ldots, u_{n,k}))\leq \exp\left(n(1+\delta)(k\lambda_{k+1}(\eta)-\phi^{k}_*(\eta))\right),
$$
then
\begin{eqnarray*}
\mu^\ba_k(R_n(w))&\leq & \mu^\ba_k(\widetilde{W}(w; u_{n,1},\ldots, u_{n,k}))\\
 &\leq & \L^k(R_n(w)) \exp\left(n\delta/2(k\lambda_{k+1}(\eta)-\phi^{k}_*(\eta))\right)=\L^k(R_n(w))\beta^n,
\end{eqnarray*}
where  $\beta:=\exp(\delta/2(k\lambda_{k+1}(\eta)-\phi^{k}_*(\eta)))\in (0,1)$. It follows that
$$C_n\subset \sigma^{-n}\circ (\pi^\ba_k)^{-1}(\Gamma_n),$$ where
$$
\Gamma_n:=\bigcup R,$$ in which the union is taken over the collection of $R\in {\mathcal R}_n$ so that $ R\cap  \pi^\ba_k(\Sigma)\neq \emptyset$ and $\mu^\ba_k(R)\leq
\L^k(R)\beta^n$. Note that
\begin{equation}\label{e-z3}
\mu^\ba_k(\Gamma_n)\leq \sum_{R\in {\mathcal R}_n:\;  R\cap  \pi^\ba_k(\Sigma)\neq \emptyset} \L^k(R)\beta^n\leq (2\ell)^k  \beta^n,
\end{equation}
where $\ell=\mbox{diam}(F^\ba)$.
Meanwhile  $A_n'\cap \Omega_n\subseteq C_n$ and $C_n\subset \sigma^{-n}\circ (\pi^\ba_k)^{-1}(\Gamma_n)$, we have
$$
A'_n\cap \Omega_n\subset  \sigma^{-n}\circ(\pi^\ba_k)^{-1}(\Gamma_n).
$$
By the invariance of $\eta$, we have $\eta(A'_n\cap \Omega_n)\leq \eta(\sigma^{-n}\circ(\pi^\ba_k)^{-1}(\Gamma_n))=\eta\circ(\pi^\ba_k)^{-1}(\Gamma_n)$.
Since $\eta\circ(\pi^\ba_k)^{-1}\ll \mu^\ba_k$ and $\lim_{n\to 0}\mu^\ba_k(\Gamma_n)= 0$ (by \eqref{e-z3}), we have
$\lim_{n\to \infty}\eta(A'_n\cap \Omega_n)=0$.
Therefore
$$
\lim_{n\to \infty}\eta\left(\bigcap_{j=n}^\infty (A'_j\cap \Omega_j)\right)=0.
$$

Note that  $\bigcap_{j=n}^\infty A'_j \subset  (\bigcap_{j=n}^\infty (A'_j\cap \Omega_j)) \cup (\Sigma\backslash  \bigcap_{j=n}^\infty \Omega_j)$, and
$\lim_{n\to \infty}\eta\left(\bigcap_{j=n}^\infty \Omega_j\right)=1$.
It follows that $\eta(\bigcap_{j=n}^\infty A'_j)=0$ and
$$
\eta(A')=\eta\left(\bigcup_{n=1}^\infty\bigcap_{j=n}^\infty A'_j\right)=0,
$$
as desired.
\end{proof}

\begin{proof}[Proof of Theorem \ref{thm-1.3}] It follows directly from Propositions \ref{pro-3.4}-\ref{pro-5.1}-\ref{pro-5.2}.
\end{proof}

\section{Extension of Falconer's formula for $q>2$ and complements to Theorem~\ref{thm-1.3}}\label{S-6}

Let $T_1,\ldots, T_m$ be non-singular linear transformations from $\R^d$ to $\R^d$ and $(p_1,\ldots, p_m)$ a probability vector. For $\ba=(a_1,\ldots, a_m)\in \R^{md}$, let $\mu^\ba$ denote the self-affine measure associated with the IFS $\{T_i+a_i\}_{i=1}^m$ and $(p_1,\ldots, p_m)$.  We begin from the following lemma.

\begin{lem}\label{extension}
Suppose that  $\|T_i\|<1/2$ for all $1\le i\le m$.     Then, for every $q>2$,  for $\L^{md}$-a.e. $\ba\in\R^{md}$, we
have $ \tau({\mu^\ba},q)\ge  \min  ((q-1) u(q),d) $, where
\begin{equation}
\label{e-ec1}
u(q)=\sup\Big \{s\geq 0: \sum_{k=0}^\infty\sum_{I\in\Sigma_k} \big(\phi^{s(q-1)}(T_{I}) \big )^{-1}p_I^{q} <\infty\Big \}.
\end{equation}
\end{lem}

\begin{proof} Fix $q> 2$. Let $s\in (0,d/(q-1))$ so that $s(q-1)$ is non-integral.  We adapt an idea used in \cite{BB06} for determining
the $L^q$-spectrum of projected measures.   Fix $\rho>0$ and $\epsilon\in (0,1)$. Let $B({\bf 0},\rho)$ stand for the closed ball of radius $\rho$ centered at ${\bf 0}$ in $\R^{md}$. Let $\mu$ denote the Bernoulli product measure on $\Sigma$ with the weight $(p_1,\ldots, p_m)$. Clearly $\mu^\ba=\mu\circ (\pi^\ba)^{-1}$. For $r>0$,
we have

\begin{eqnarray*}
&& \int_{B({\bf 0},\rho)}\int\mu^\ba(B(z,r))^{q-1}\mathrm{d}\mu^\ba(z)\mathrm{d}\ba=\int_{B({\bf 0},\rho)}\int_{\Sigma}\mu^\ba(B(\pi^\ba x,r))^{q-1}\mathrm{d}\mu(x)\,\mathrm{d}\ba\\
&\mbox{}&\quad=\int_{\Sigma}\int_{B({\bf 0},\rho)}\Big (\int_{\Sigma}\mathbf{1}_{\{|\pi^\ba y-\pi^\ba x|\le r\}}\mathrm{d}\mu(y)\Big)^{q-1}\mathrm{d}\ba\, \mathrm{d}\mu(x)\\
&\mbox{}&\quad\le
\int_{\Sigma}\int_{B({\bf 0},\rho)}\Big (\int_{\Sigma}\frac{r^s}{|\pi^\ba y-\pi^\ba x|^s}\mathrm{d}\mu(y)\Big)^{q-1}\mathrm{d}\ba\, \mathrm{d}\mu(x)\\
&\mbox{}&\quad \le \int_{\Sigma} \Big
(\int_{\Sigma}\Big (\int_{B({\bf 0},\rho)}   \frac{r^{s(q-1)}}{|\pi^\ba y-\pi^\ba x|^{s(q-1)}}  \mathrm{d}\ba\Big )^{1/(q-1)}    \mathrm{d}\mu(y)  \Big)^{q-1}  \mathrm{d}\mu(x),
\end{eqnarray*}
where we use Minkowski's inequality in the last inequality.
By  \cite[Lemma 2.1]{Fal99}, $$\int_{B({\bf 0},\rho)} \frac{1}{|\pi^\ba y-\pi^\ba x|^{s(q-1)}}  \mathrm{d}\ba\leq \frac{C}{\phi^{s(q-1)}(T_{x\land y})}$$ for some  $C=C(\rho,s(q-1)) >0$. Hence we have
\begin{eqnarray*}
&&\int_{B({\bf 0},\rho)}\int\mu^\ba(B(z,r))^{q-1}\mathrm{d}\mu^\ba(z)\mathrm{d}\ba\\
&\mbox{}&\quad \le Cr^{s(q-1)}  \int_{\Sigma} \Big (\int_{\Sigma}\big(\phi^{s(q-1)}(T_{x\land y}) \big )^{-1/(q-1)}\mathrm{d}\mu(y)
\Big)^{q-1} \mathrm{d}\mu(x) \\
&\mbox{}&\quad \le  Cr^{s(q-1)}  \int_{\Sigma}  \Big (\sum_{k=0}^\infty \big(\phi^{s(q-1)}(T_{x|k}) \big )^{-1/(q-1)}\mu([x|k]) \Big)^{q-1}
\mathrm{d}\mu(x)\\
&\mbox{}&\quad \le  M Cr^{s(q-1)}  \int_{\Sigma}  \Big (\sum_{k=0}^\infty \big(\phi^{s(q-1)}(T_{x|k}) \big )^{-1}\mu([x|k])^{(q-1)(1-\epsilon)} \Big)
\mathrm{d}\mu(x)\\
&\mbox{}& \qquad\qquad\qquad(\text{by H\"{o}lder's inequality})\\
&\mbox{}&\quad = M Cr^{s(q-1)}  \sum_{k=0}^\infty\sum_{I\in\Sigma_k} \big(\phi^{s(q-1)}(T_{I}) \big )^{-1}\mu([I])^{q-(q-1)\epsilon} ,
\end{eqnarray*}
where
$
M=\sup_{x\in\Sigma} \Big (\sum_{k=0}^\infty \mu([x|k])^{\epsilon(q-1)/(q-2)}\Big )^{(q-2)/(q-1)}<\infty$.

Now, let $0<s_1<s_0$. Set $\gamma=(s_0-s_1)(q-1)$. Given $\epsilon'>0$, for each $I\in\Sigma^*$ such that $\mu([I])>0$, we have
$$
\frac{\big(\phi^{s_1(q-1)}(T_{I}) \big )^{-1}\mu([I])^{q-\epsilon'}}{\big(\phi^{s_0(q-1)}(T_{I}) \big )^{-1}\mu([I])^{q}}=\frac{\phi^{s_0(q-1)}(T_{I})}{\phi^{s_1(q-1)}(T_{I})}\mu([I])^{-\epsilon'}\le \alpha_1(T_I)^{\gamma}\mu([I])^{-\epsilon'}\le  (2^{-\gamma }c^{-\epsilon'})^{|I|},
$$
where $c=\min_{1\leq i\leq m}p_i$.
Suppose that $\epsilon'$ is so small that $2^{-\gamma }c^{-\epsilon'}<1$ and set $\epsilon=\epsilon'/(q-1)$. If $s_0< \min (u(q),d/(q-1))$ and $s_1(q-1)$ is not an integer, we deduce from the above estimates that
$$
\sup_{r>0}\frac{\int_{B({\bf 0},\rho)}\int\mu^\ba(B(z,r))^{q-1}\mathrm{d}\mu^\ba(z)\mathrm{d}\ba}{r^{s_1(q-1)}}<\infty.
$$
This implies that for all $s'_1<s_1$,
$$
\int_{B({\bf 0},\rho)}\sum_{n\ge 1} \frac{\int\mu^\ba(B(z,2^{-n}))^{q-1}\mathrm{d}\mu^\ba(z)}{2^{-ns'_1(q-1)}}\mathrm{d}\ba<\infty,
$$
hence, for $\L^{md}$-almost every $\ba\in B({\bf 0},\rho)$, we have
$$
\liminf_{n\to \infty}\frac{-1}{n\log(2)}{\log \int\mu^\ba(B(z,2^{-n}))^{q-1}\mathrm{d}\mu^\ba(z)}\ge s'_1(q-1).
$$
Moreover, the left hand side in the previous inequality is nothing but $\tau(\mu^\ba,q)$. Since $s'_1$ and $s_1$ can be taken arbitrarily close to  $\min (u(q),d/(q-1))$ (as long as $s_1(q-1)$ is not an integer) and $\rho$ is arbitrary, we get the desired lower bound for $\tau(\mu^\ba,q)$.

\end{proof}

Let $D(\cdot)$ and $\tau(\cdot)$ be defined as in \eqref{e-1.1}-\eqref{e-1.2}. By Lemma \ref{extension},  we can extend Falconer's formula of $\tau(\mu^\ba,q)$  as follows.

\begin{thm}\label{corextension}
Suppose that  $\|T_i\|<1/2$ for all $1\le i\le m$.
\begin{enumerate}
\item For $\L^{md}$-a.e. $\ba\in\R^{md}$ we have $\tau({\mu^\ba},q)=\tau(q)$ for all $q$ in the following set
\begin{equation}
\label{e-ec}
[2,\;\sup\{t:\; D(t)/(t-1)\le 1, \ \tau(t)\le 1\}].
\end{equation}
 This set is a non-empty interval for instance if $\tau'(1+)\le 1$, in which case it contains $[2,1+1/\tau'(1+)]$.

\item If the $T_i$ are similitudes, then for $\L^{md}$-a.e. $\ba\in\R^{md}$ we have $\tau({\mu^\ba},q)=\tau(q)$ for all $q\in [2,\max\{q:\tau(q)\le d\}]$.
\end{enumerate}
\end{thm}
\begin{proof}
By continuity of the functions $\tau({\mu^\ba},\cdot)$ and $\tau(\cdot)$, it is enough to prove the result for a
fixed $q$ and $\L^{md}$-almost every $\ba$.

(1) Let $q$ be a point in the interval given as in \eqref{e-ec}. Since $q-1\ge 1$, $D(q)/(q-1)\le 1$ implies that $D(q)=\tau(q)\le 1$. Thus  $\max (D(q), D(q)/(q-1))\le 1$, so for all $0<s\le D(q)$ and $I\in\Sigma^*$ we have  $\phi^{s(q-1)}(T_{I})=(\phi^{s}(T_{I}))^{q-1}$ by definition of the singular value functions $\phi^s$. Hence
$(q-1)u(q)=D(q)$, where $u(q)$ is defined as in \eqref{e-ec1}. Therefore   $\tau(q)=D(q)=(q-1)u(q)$. This gives the conclusion thanks to Lemma~\ref{extension} and  Lemma \ref{lem-5.1}. Finally, if $\tau'(1+)\le 1$ and $q\le 1+1/\tau'(1+)$, by concavity of $\tau$ we have $\tau(q)\le \tau'(1+)(q-1)\le 1$, and also we have $\tau(q)/(q-1)=D(q)/(q-1)\le 1$.

(2) Let $q\geq 2$ so that $\tau(q)\le d$. Since $T_i$ are similitudes, we have $\phi^{s(q-1)}(T_{I})=(\phi^{s}(T_{I}))^{q-1}$ for all $I\in\Sigma^*$ and $s>0$. By
\eqref{e-ec1}, $(q-1)u(q)=D(q)$.  Since $\tau(q)\leq d\leq d(q-1)$, we have $\tau(q)=D(q)=(q-1)u(q)$. By Lemma \ref{extension}, $\tau_{\mu^\ba}(q)\ge \min (\tau(q),d)=\tau(q)$  for $\L^{md}$-almost all $\ba\in \R^{md}$.  This together with Lemma \ref{lem-5.1} yields the desired result.
\end{proof}

As an application of Theorem \ref{corextension}, we have the following two theorems.

\begin{thm}\label{thm-6.3}
The conclusions of Theorem~\ref{thm-1.3}(ii) extend to those $q\ge 2$ such that $D(q)< q-1$ and $\tau(q)< 1$.
\end{thm}

\begin{thm}\label{thm-6.4}Suppose that the maps $T_i$ ($1\le i\le m$) are similitudes with  $\|T_i\|<1/2$. Denote  $q_{\max}=  \max (2,\max\{q>0:\tau(q)\le d\})$.
Then the following properties hold.
\begin{enumerate}
\item For all $q>0$, $D(q)$ is the analytic solution of the equation $\sum_{i=1}^m p_i^q\|T_i\|^{-t}=1$.

\item Suppose $D'(1)\ge d$. Let $s=\inf\{D(q)/(q-1):1<q\le 2\}$.

\begin{itemize}
\item If $s\ge d$, then $q_{\max}=2$ and for $\L^{md}$-a.e. $\ba\in\R^{md}$:
$$
\tau(\mu^\ba,q)=d(q-1)\text { for all $q\in[0,q_{\max}]$}.
$$
\item If $s<d$ then $q_{\max}>2$ and for $\L^{md}$-a.e. $\ba\in\R^{md}$:
$$
\tau(\mu^\ba,q)=\begin{cases}d(q-1)& \text{if }q\in [0,q_{\min}),\\
 D(q)&\text{if $q\in[q_{\min},q_{\max}]$,}
 \end{cases}
 $$
 where  $q_{\min}=\inf\{q>1: D(q)/(q-1)<d\}$; moreover, the multifractal formalism holds for $\mu^\ba$ at all $\alpha\in \{d\}\cup [D'(q_{\max}),D'(q_{\min})]$.
\end{itemize}
\smallskip

\item If $D'(1)<d$, then $q_{\max}>2$. Let  $\tilde{q}_{\min}=\inf\{q>0: D'(q)q-D(q)\le d\}$. For $\L^{md}$-a.e. $\ba\in\R^{md}$, $$
 \tau(\mu^\ba,q)=
 \begin{cases} \displaystyle\frac{d+D(\tilde{q}_{\min})}{\tilde{q}_{\min}}q-d&\text{ if $q\in [0,\tilde{q}_{\min})$},\\
 D(q)&\text{if $q\in[\tilde{q}_{\min},q_{\max}]$.}
 \end{cases}
 $$
 Moreover, the multifractal formalism holds for $\mu^\ba$ at all $\alpha\in  [D'(q_{\max}),D(1)]$. Also, for each $\alpha\in (D'(1),D'(\tilde{q}_{\min})]$, for $\L^{md}$-a.e. $\ba\in \R^{md}$, the multifractal formalism holds at  $\alpha$.
\end{enumerate}
\end{thm}

\begin{rem}\label{rem-6.5}
{\rm (1)} By \cite{Fen07} we know that for all $\ba\in \R^{md}$, the self-similar measure $\mu^\ba$ obeys the multifractal formalism at each $\alpha$ of the form $\tau'(\mu^\ba,q)$, with $q>1$. Moreover, the measure $\mu^\ba$ is exact dimensional by \cite{FeHu09}, so the multifractal formalism holds at $\alpha=\dim_H\, \mu^\ba$.  Theorem~\ref{thm-6.4} gives precisions on the value of the $L^q$-spectrum and the validity of the multifractal formalism. When $D'(1)>d$ and $\inf\{D(q)/(q-1):1<q\le 2\}<d$, for $\L^{md}$-a.e. $\ba\in\R^{md}$ the measure $\mu^\ba$ is absolutely continuous with respect to Lebesgue measure and has a non trivial $L^q$-spectrum. This fact is already noticed in \cite{Fen11}.

{\rm (2)} Theorem~\ref{thm-6.4} takes a form similar to that of the result obtained in \cite{BB11} for the orthogonal projections of Gibbs measures on $\R^d$ to almost every linear subspace of a given dimension between $1$ and $d$, when $d\ge 2$.
\end{rem}

\medskip

\begin{proof}[Proof of Theorem~\ref{thm-6.3}]  The proof is similar to the proof of Theorem~\ref{thm-1.3}(i), except that we already know the value of $\tau(\mu^\ba,q)$ thanks to  Theorem~\ref{corextension}.
\end{proof}

\medskip

\begin{proof}[Proof of Theorem~\ref{thm-6.4}] (1) This is clear.

(2) If $D'(1)> d$, then by Theorem~\ref{thm-1.2}, for $\L^{md}$-a.e. $ \ba\in\R^{md}$ we have  $\tau(\mu^\ba,q)=d(q-1)$ on a neighborhood of $1+$; if $D'(1)=d$, either $D$ is linear equal to $d(q-1)$, or it is strictly concave and still by Theorem~\ref{thm-1.2}, for $\L^{md}$-a.e. $ \ba\in\R^{md}$ we have  $ \tau(\mu^\ba,q)=D(q)$ on a neighborhood of $1+$. Consequently, in both cases $ \tau'(\mu^\ba,1+)=d$, so since $\tau(\mu^\ba,\cdot)$ is concave $\tau(\mu^\ba,0)\ge -d$ and $\tau(\mu^\ba,1)=0$, we must have $\tau(\mu^\ba,q)=d(q-1)$ over $[0,1]$.

Now, if $s\ge d$ then $D(q)\ge d(q-1)$ for all $q\in (1,2]$, so by Theorem~\ref{thm-1.2}, for $\L^{md}$-a.e. $ \ba\in\R^{md}$ we have  $\tau(\mu^\ba,q)=d(q-1)=\tau(q)$ for $q\in[1,2]$, hence $q_{\max}=2$.

If $s<d$, we have $\tau(2)=D(2)<d (2-1)=d$, so $q_{\max}>2$. The value of $\tau(\mu^\ba,\cdot)$ over $[1,q_{\min})$ and $[q_{\min},q_{\max}]$ is obtained again thanks to Theorems~\ref{thm-1.2} and \ref{corextension}.

For the validity of the multifractal formalism, at $\alpha=d$ it comes from the fact that $\tau'(\mu^\ba,1)$ exists and equals $d$ (see \cite{Ngai97}).

Since $\tau(\mu^\ba,\cdot)$ coincides with $D$ and $\tau$ over the open interval $(q_{\min},q_{\max})$, we can use \cite{Fen07} and Remark~\ref{rem-6.5} to have the validity of the multifractal formalism, for $\L^{md}$-a.e. $\ba\in\R^{md}$,  for all  $\alpha \in (D'(q_{\max}),D'(q_{\min}))$.  If $\alpha=D'(q_{\min})=\tau'(\mu^\ba,q_{\min}+)$, we have $\alpha q_{\min}-D(q_{\min})\le d$ and we can use the same proof as that of Theorem~\ref{thm-1.3}(ii) when $k=0$, since now the singular values function $\phi^s(T)$ simplifies to be  $\alpha_1(T)^s$ for all $s>0$ and is multiplicative. We can do the same at $\alpha=D'(q_{\max})$.

(3) By concavity of $D$, we have $\tau(q)=D(q)\le D(1)(q-1)<d(q-1)$ for all $q>1$, so $q_{\max}>2$. Moreover, by using Theorem~\ref{thm-1.2} as above we get that  for $\L^{md}$-a.e. $ \ba\in\R^{md}$, we have  $ \tau(\mu^\ba,q)=D(q)$ on $[1,q_{\max}]$. The validity of the multifractal formalism over $[D'(q_{\max}),D'(1)]$ is obtained as above over $[D'(q_{\max}),D'(\tilde{q}_{\min}]$.

The inequality $D'(1)<d$ also implies  $\tilde{q}_{\min}\in [0,1)$. Moreover, if $q\in (\tilde{q}_{\min},1)$, by concavity of $D$, $D'(q)q-D(q)<d$ implies that  $D(q)>d(q-1)$, so that $\tau(q)=D(q)$; consequently, by Lemma~\ref{lem-5.1} we have $ \tau(\mu^\ba,q)\ge D(q)$  for $\L^{md}$-a.e. $ \ba\in\R^{md}$, for all $q\in [\tilde{q}_{\min},1)$. Then, we can use the same argument as that used to prove Theorem~\ref{thm-1.3}(i) (noting again that the singular values function simplifies to be $\alpha_1(T)^s$) to get that for each $\alpha=D'(q)$, $q\in  [\tilde{q}_{\min},1)$, we have $ \alpha q-D(q)\le \dim E(\mu^\ba,D'(q))\le \alpha q-\tau^\ba (q)\le \alpha q-D(q)$, for $\L^{md}$-a.e. $ \ba\in\R^{md}$.

This yields that for $\L^{md}$-a.e. $ \ba\in\R^{md}$,  $\tau(\mu^\ba,q)= D(q)$ for all $q\in [\tilde{q}_{\min},1]$. Now, if $\tilde{q}_{\min}>0$, then by definition of $\tilde{q}_{\min}$ the tangent to $D$ at $(\tilde{q}_{\min},D(\tilde{q}_{\min}))$ crosses the $y$-axis at $(0,-d)$, so since $\tau(\mu^\ba,\cdot)$ is concave and $\tau(\mu^\ba,0)\ge -d$, $\tau(\mu^\ba,\cdot)$ must take the linear expression of the statement over $[0,\tilde{q}_{\min})$.
\end{proof}

In the remainder  of this section, we  provide a formula of the $L^q$-spectrum  for certain ``almost all'' non-overlapping planar self-affine measures over a range $\supseteq [0,2]$.

\begin{de}
Following \cite{FeWa05}, we say that an IFS $\{S_i\}_{i=1}^m$ on $\R^2$ satisfies the {\it rectangular open set  condition} (ROSC)  if there exists an open rectangle $R=(0, r_1)\times (0, r_2)+v$ such that $S_i(R)$ ($1\leq i\leq m$) are disjoint subsets of $R$.
\end{de}

\begin{ex}\label{e-6.C} Assume that  $T_1=T_2=\ldots=T_m={\rm diag} (t_1, t_2)$ with $1/2>t_1>t_2$. Let ${\bf p}=(p_1,\ldots, p_m)$ be a probability vector.  For $$\bc=((a_1, b_1), \ldots, (a_m, b_m))\in \R^{2m},$$
let $\mu^\bc$ denote the self-affine measure associated with the IFS $\{T_i+(a_i, b_i)\}_{i=1}^m$ on $\R^2$ and  the probability vector ${\bf p}$.
Denote by $V$ the set of points $\bc\in \R^{2m}$ so that  $\{T_i+(a_i, b_i)\}_{i=1}^m$  satisfies the ROSC.  By \cite[Theorem 2]{FeWa05},
for any $\bc\in V$,
\begin{equation}
\label{e-6.d}
\tau(\mu^\bc, q)=\tau(\nu^\ba, q)\left(1-\frac{\log t_1}{\log t_2}\right)+\frac{\log \sum_{i=1}^m p_i^q}{\log t_2}, \qquad\; \forall\; q>0,
\end{equation}
where $\nu^\ba$ denotes the self-similar measure associated with the IFS $\{t_1x+a_i\}_{i=1}^m$ and ${\bf p}$,  $\tau(\nu^\ba, q)$ denotes the $L^q$-spectrum of $\nu^\ba$. Denote by
$B(q)={\log \sum_{i=1}^m p_i^q}/{\log t_1}$. Let $q_{\max}=\max\{2, q_1\}$ where $q_1$ is the unique positive number satisfying  $B(q_1)=t_1$.
 By Theorem~\ref{thm-6.4},  if $B'(1)\geq 1$, then for $\L^m$-a.e $\ba\in \R^m$, $\tau(\nu^\ba, q)=q-1$ for every $0\leq q\leq 1$; meanwhile if  $B'(1)< 1$, then
 for $\L^m$-a.e $\ba\in \R^m$,
 $$
 \tau(\nu^\ba, q)=\left\{\begin{array}{ll}B'(q_0)q-1 &\mbox{ if } q\in [0, q_0],\\
 B(q) &\mbox{ if } q\in (q_0,1],
 \end{array}\right.
 $$
 where $q_0:=\inf\{q>0: B'(q)q-B(q)\leq 1\}$. Furthermore, by Theorem \ref{corextension}, we have for $\L^m$-a.e $\ba\in \R^m$,
 \begin{equation}
 \label{e-6.e}\tau(\nu^\ba, q)=\max\{B(q), q-1\},\quad \forall \; q\in (1, q_{\max}).
 \end{equation}
  Now one obtains the exact formula of  $\tau(\mu^\bc, q)$ by \eqref{e-6.d} for  $\L^m$-a.e $\bc\in V$ and every $q\in [0,q_{\max}]$.

 \begin{rem} According to the formula \eqref{e-6.e}  in Example \ref{e-6.C}, it is easy to  see that for each $q\in (1,2)$, one can choose $m\in \N$, $t_1\in (0,1/2)$ and  ${\bf p}=(p_1,\ldots, p_m)$ so that for $\L^m$-a.e $\ba\in \R^m$, $\tau(\nu^\ba, q)$ is not differentiable at $q$. Hence for any $q\in (1,2)$, there exists a self-similar measure on $\R$ whose the $L^q$ spectrum is not differentiable at $q$.
 \end{rem}

\end{ex}

\section{Final remarks}\label{S-7}

In this section we first give two remarks about the extensions of our results.

\begin{itemize}

\item[(i)]
All the results presented in this paper hold if we replace the Bernoulli measure $\mu$ by a Gibbs measure associated to a potential satisfying the bounded distortion property. This is due to the  almost multiplicative property of such a measure. The corresponding expression of $D(q)$ can be found in \cite{Fal99}.

\medskip

\item[(ii)]
 Our results can be partially extended to the projections of Bernoulli measures and Gibbs measures on the model of randomly perturbed self-affine attractors introduced in \cite{JPS07}. For such a construction, the condition $\|T_i\|<1/2$ for all $1\le i\le m$ can be relaxed to $\|T_i\|<1$ for all $1\le i\le m$.  Moreover, Falconer's formula extends to $[2,\infty)$ \cite{Fal10}. Then, mimicking the proofs written in the present paper, Theorem~\ref{thm-1.3}(i) holds as well as Theorem~\ref{thm-1.3}(ii) for all $q>2$ under the constraint that $k=0$. We don't know whether this extension can  pass to $k>0$, because it seems non trivial to transpose the arguments developed in Proposition~\ref{pro-4.1} and Lemma~\ref{lem-5.7} in relation with the equivalence to the Lebesgue measure for the measures under consideration. In the special case of almost self-similar measures, the validity of Falconer's formula over $[2,\infty)$ implies that the results of Theorem~\ref{thm-6.4} hold if, when $D'(1)\ge 1$, one sets $q_{\max}=\infty$ and $s=\inf\{D(q)/(q-1):q>1\}$.

\end{itemize}

In the end, we point out that in a related paper \cite{JoSi07} Jordan and Simon  studied the multifractal structure  of  Birkhoff averages on almost all self-affine sets.

\appendix
\section{Concavity properties of the functions $D$ and $\tau$. }\label{S-8}
It follows from the study of the $L^q$-spectrum of almost self-affine measures achieved in \cite{Fal10} that $\tau$ is concave over $(1,\infty)$. However, this fact is not obtained directly from the definition of $D(q)$.  Our Theorem~\ref{thm-1.3}(i) requires concavity properties of $D$ for $q\in (0,1)$ which cannot be reached by the approach used in \cite{Fal10}. In the following we provide a proof of these properties, and for the sake of completeness, a direct proof of the concavity of $\tau$ over $(1,\infty)$.

\begin{pro}\label{Dconcavity}
The mapping $D$ is concave over the intervals of those $q\neq 1$ such that $D(q)/(q-1)\in(k,k+1)$ for some integer $0\le k\le d-1$.
\end{pro}

\begin{pro}\label{tauconcavity}
The mapping $\tau$ is concave over $(1,\infty)$.
\end{pro}
\noindent
{\it Proof of Proposition~\ref{Dconcavity}.}
 It is clear from \eqref{e-3.1} and the fact that both $p_I$ and $\phi^s$ $(s> 0)$ are bounded away from 0 and $\infty$ by geometric sequences that $D(q)$ is continuous. So if $0\le k\le d-1$ is an integer, the set $J_k$ of those $q\in (0,1)$ such that  $D(q)/(q-1)\in(k,k+1)$ is an interval, as well as the set $J'_k$ of those $q\in(1,\infty)$ with the same property.

 Let us deal with $J_k$. The case of $J'_k$ is similar.  Fix $q,q'\in J_k$ and $\lambda\in (0,1)$. Pick $s, s'$ so that   $D(q)/(q-1)<s<k+1$, $D(q')/(q'-1)<s'<k+1$.
 Then
 \begin{equation}\label{eq0}
 \begin{split}
  &\limsup_{n\to\infty}\frac{1}{n}\log \sum_{I\in\Sigma_n}\phi^{s}(T_I)^{1-q}p_I^{q}\le 0,\\
  &\limsup_{n\to\infty}\frac{1}{n}\log \sum_{I\in\Sigma_n}\phi^{s'}(T_I)^{1-q'}p_I^{q'}\le 0,
  \end{split}
 \end{equation}
 Define
 $$
 q_\lambda =(1-\lambda)q+\lambda q ',\quad s_\lambda = \frac{(1-\lambda)(q-1)s+\lambda(q'-1)s'}{q_\lambda-1}.
 $$ If we prove that
 \begin{equation}\label{eq}
 \limsup_{n\to\infty}\frac{1}{n}\log \sum_{I\in\Sigma_n}\phi^{s_\lambda}(T_I)^{1-q_\lambda}p_I^{q_\lambda}\le 0,
 \end{equation}
then by definition of $D(q_\lambda)$, we have
$$
\frac{D(q_\lambda)}{q_\lambda-1}\le  s_\lambda =  \frac{(1-\lambda)(q-1)s+\lambda(q'-1)s'}{q_\lambda-1}$$
for all $s,s'$ has above, so
$$
D(q_\lambda)\ge  (1-\lambda)D(q)+\lambda D(q').
$$
Now we prove \eqref{eq}.  By construction we have  $k<s_\lambda<k+1$, so
\begin{eqnarray*}
\sum_{I\in\Sigma_n}\phi^{s_\lambda}(T_I)^{1-q_\lambda}p_I^{q_\lambda}&=&\sum_{I\in\Sigma_n}\big (\phi^{s}(T_I)^{1-q}p_I^{q}\big )^{1-\lambda}\big (\phi^{s'}(T_I)^{1-q'}p_I^{q'}\big )^{\lambda}\\
&\le &\Big( \sum_{I\in\Sigma_n}\phi^{s}(T_I)^{1-q}p_I^{q}\Big)^{1-\lambda} \Big (\sum_{I\in\Sigma_n}\phi^{s'}(T_I)^{1-q'}p_I^{q'}\Big)^{\lambda},
\end{eqnarray*}
where the second inequality comes from the H\"{o}lder's inequality.
This together with \eqref{eq0} yields \eqref{eq}. $\square$

\begin{lem}\label{lem-7.3}
Let $q_0>1$ such that $D(q_0)/(q_0-1)=k$ for some integer $k\in\{1,2,\ldots,d\}$. Then
$$
\frac{D(q)}{q-1}\le k\text{ if }q>q_0\quad \text{and}\quad \frac{D(q)}{q-1}\le k\text{ if }q<q_0.
$$
\end{lem}
\begin{proof} First assume that $q>q_0$. To show that $D(q)/(q-1)\le k$, it suffices to show that
\begin{equation}\label{eq-7.1}
\forall \, \delta>0, \ \sum_{I\in\Sigma_n}\phi^k(T_I)^{1-q}p_I^q\ge e^{-n\delta}\quad \text{ for large enough }n.
\end{equation}
Assume that \eqref{eq-7.1} does not hold., i.e. there exists $\delta>0$ such that
$$
 \sum_{I\in\Sigma_n}\phi^k(T_I)^{1-q}p_I^q< e^{-n\delta}\quad  \text{ infinitely often (i.o). }
$$
Note that $\sum_{I\in\Sigma_n}p_I=1$. Take $\lambda\in (0,1)$ such that $(1-\lambda)q+\lambda=q_0$. Then, by the H\"older inequality
$$
 \sum_{I\in\Sigma_n}\phi^k(T_I)^{(1-\lambda)(1-q)}p_I^{(1-\lambda)q}p_I^\lambda\le e^{-n(1-\lambda)}\cdot 1^\lambda \quad \text{ i.o.}
 $$
 i.e.
 $$
  \sum_{I\in\Sigma_n}\phi^k(T_I)^{1-q_0}p_I^{q_0}\le e^{-n(1-\lambda)}\quad \text{ i.o.},
  $$
  a contradiction with our assumption on $D(q_0)/(q_0-1)$.

Next assume that $q<q_0$. To show that $D(q)/(q-1)\ge k$, it suffices to show that
\begin{equation*}
\forall \, \delta>0, \ \sum_{I\in\Sigma_n}\phi^k(T_I)^{1-q}p_I^q\le e^{n\delta}\quad \text{ for large enough }n.
\end{equation*}
To see this, since $D(q_0)/(q_0-1)=k$, we have
$$
\sum_{I\in\Sigma_n}\phi^k(T_I)^{1-q_0}p_I^{q_0}\le e^{n\delta}\quad \text{ for large enough }n.
$$
Take $\lambda\in (0,1)$ such that $(1-\lambda)q_0+\lambda=q$. Then, by the H\"older inequality
$$
 \sum_{I\in\Sigma_n}\phi^k(T_I)^{(1-\lambda)(1-q_0)}p_I^{(1-\lambda)q_0}p_I^\lambda\le e^{n(1-\lambda)}\cdot 1^\lambda
 $$
 i.e.
 $$
  \sum_{I\in\Sigma_n}\phi^k(T_I)^{(1-q_0)}p_I^{q_0}\le e^{n(1-\lambda)},
  $$
if $n$ is large enough,  as desired.
  \end{proof}

\begin{rem}
The same argument (with $k$ replaced by any positive number $s$ shows that $q\mapsto D(q)/(q-1)$ is non-increasing on $(1,\infty)$.
\end{rem}

\noindent
{\it Proof of Proposition~\ref{tauconcavity}.} Due to Proposition~\ref{Dconcavity}, it suffices to show that

(1) If $\displaystyle \frac{D(q_0)}{q_0-1}\in\{1,2,\ldots,d-1\}$ for some $q_0>1$, then  $D'(q_0+)\le D'(q_0-)$.

(2) If $\displaystyle \frac{D(q_0)}{q_0-1}=d$ for some $q_0>1$, then $D'(q_0+)\le d$ (by Lemma~\ref{lem-7.3}, $\tau(q)=d(q-1)$ if $1<q<q_0$).

Let us first prove (1). Assume on the contrary that (1) does not hold, i.e. $D'(q_0+)>D'(q_0-)$. Then there exists a small $\epsilon>0$ such that
$$
D(q_0)<\frac{1}{2} D(q_0+\epsilon)+\frac{1}{2} D(q_0-\epsilon),
$$
and
$$
\frac{D(q_0+\epsilon)}{q_0+\epsilon-1}\le k\le \frac{D(q_0-\epsilon)}{q_0-\epsilon-1}< k+1\quad (\text{by Lemma~\ref{lem-7.3}}).
$$
Let $s_1=\displaystyle \frac{D(q_0+\epsilon)}{q_0+\epsilon-1}$ and $s_2=\displaystyle \frac{D(q_0-\epsilon)}{q_0-\epsilon-1}$, $q_1=q_0+\epsilon$, $q_2=q_0-\epsilon$. Then, for all $\delta>0$ and $i\in \{1,2\}$,
$$
\sum_{I\in\Sigma_n}\phi^{s_i}(T_I)^{1-q_i}p_I^{q_i}\le e^{n\delta}\quad \text{ for large enough }n.
$$
By the H\"older inequality, we have
$$
\sum_{I\in\Sigma_n}\phi^{s_1}(T_I)^{(1-q_1)/2}p_I^{q_1/2}\phi^{s_2}(T_I)^{(1-q_2)/2}p_I^{q_2/2}\le e^{n\delta},
$$
i.e.
$$
\sum_{I\in\Sigma_n}\phi^{s_1}(T_I)^{(1-q_1)/2}\phi^{s_2}(T_I)^{(1-q_2)/2}p_I^{q_0}\le e^{n\delta},
$$
Note that
\begin{eqnarray*}
&&\phi^{s_1}(T_I)^{(1-q_1)/2}\phi^{s_2}(T_I)^{(1-q_2)/2}\\
&=&(\alpha_1\alpha_2\cdots\alpha_k)^{(1-q_1)/2}\alpha_k^{(s_1-k)(1-q_1)/2} \cdot (\alpha_1\alpha_2\cdots\alpha_k)^{(1-q_2)/2}\alpha_{k+1}^{(s_2-k)(1-q_2)/2} \\
&& \text{(where $\alpha_i=\alpha_i(T_I)$)}\\
&=&(\alpha_1\alpha_2\cdots\alpha_k)^{1-q_0}\alpha_k^{(s_1-k)(1-q_1)/2}\alpha_{k+1}^{(s_2-k)(1-q_2)/2}\\
&\ge &(\alpha_1\alpha_2\cdots\alpha_k)^{1-q_0}\alpha_{k+1}^{(s_1-k)(1-q_1)/2}\alpha_{k+1}^{(s_2-k)(1-q_2)/2} \quad \text{(using $(s_1-k)(1-q_1)\ge 0$)}\\
&=&(\alpha_1\alpha_2\cdots\alpha_k)^{1-q_0} \alpha_{k+1}^{-\frac{D(q_1)+D(q_2)}{2}-k(1-q_0)}\\
&\ge &(\alpha_1\alpha_2\cdots\alpha_k)^{1-q_0} \alpha_{k+1}^{-(D(q_0)+\gamma)-k(1-q_0)}\quad (\text{with $\gamma \gg\delta$})\\
&=&(\alpha_1\alpha_2\cdots\alpha_k)^{1-q_0} \alpha_{k+1}^{-\gamma}\\
&\ge& \phi^k(T_I)^{1-q_0}\cdot e^{n\gamma'}\quad (\text{with $\gamma' \gg\delta$})
\end{eqnarray*}
Therefore,
$$
\sum_{I\in\Sigma_n}\phi^{k}(T_I)^{(1-q_0)}p_I^{q_0}\le e^{-n(\gamma'-\delta)}\quad (\text{with $\gamma' \gg\delta$})
$$
for large enough $n$, a contradiction. This proves (1).

Next we show (2). To see this, recall that $D(q_0)/(q_0-1)=d$ and $D(q)/(q-1)\le d$ if $q>q_0$.
Now, since $D(q)/(q-1)$ is non increasing over $(1,\infty)$, either  $D(q)/(q-1)=d$ in a right neighborhood of $q_0$, or $D(q)/(q-1)<d$ for all $q>q_0$, and by Proposition~\ref{Dconcavity} $D$ is concave on a right neighborhood of $q_0$. Thus the inequality  $D(q)/(q-1)\le d$ for  $q>q_0$ implies $D'(q_0+)\le d$.

\noindent{\bf Acknowledgements}. Feng was partially supported by the RGC grant  and the Focused Investments Scheme  in CUHK.

\end{document}